\documentclass[a4paper,11pt]{amsart}
\usepackage[ngerman,english]{babel}
\usepackage[T1]{fontenc}
\usepackage[ansinew]{inputenc}

\usepackage{amsmath, amsfonts, amssymb, amsthm}
\usepackage{mathrsfs}
\usepackage{mathtools}

\usepackage[shortalphabetic,initials]{amsrefs}
\usepackage{esint}
\usepackage{nccfoots}
\usepackage{accents}

\usepackage{color}
\usepackage{enumerate}

\numberwithin{equation}{section}

\newcommand{\mm}{\mathfrak{m}}
\newcommand{\dd}{\operatorname{d}\!}
\newcommand{\dm}{\mathop{}\!\mathrm{d}\mathfrak{m}} %Integral dm

\newcommand{\eent}{\operatorname{Ent}} % Entropy w/o argument
\newcommand{\ch}{\operatorname{Ch}} % Cheeger energy w/o argument
\newcommand{\rcd} [2]{\operatorname{RCD}^*(#1,#2)} %RCD(K,N) spaces
\newcommand{\sob}{W^{1,2}} %Sobolev short
\newcommand{\li}{\operatorname{lip}} %local Lipschitz constant
\newcommand{\Lip}{\operatorname{Lip}} 
\newcommand{\lip}{\operatorname{Lip}} %Lipschitz functions --> dense in \sob
\newcommand{\ee}{\operatorname{e}} % evaluation map
\newcommand{\ric}{\operatorname{Ric}}

\newcommand{\ehat}{\hat{\mathcal E}}
\newcommand{\chhat}{\widehat\ch}
\newcommand{\kringel}[1]{\accentset{\circ}{#1}}
\newcommand{\hatp}{\hat{\mathscr P}\!}
\newcommand{\scrp}{\mathscr P}

\newcommand{\tild}{\tilde}
\newcommand{\evik}{\operatorname{EVI}_K}

\newcommand{\mn}{\mathbb N}

\newcommand{\mr}{\mathbb R}

\newcommand{\eps}{\varepsilon}

\theoremstyle{definition}
\newtheorem{Def}{Definition}\numberwithin{Def}{section}

\theoremstyle{plain}
\newtheorem{Thm}[Def]{Theorem}

\newtheorem{Prop}[Def]{Proposition}
\newtheorem{Lemma}[Def]{Lemma}
\newtheorem{Cor}[Def]{Corollary}

\theoremstyle{plain}
\newtheorem{theorem}[Def]{Theorem}
\newtheorem{corollary}[Def]{Corollary}
\newtheorem{lemma}[Def]{Lemma}
\newtheorem{proposition}[Def]{Proposition}

\newtheorem{assumption}[Def]{Assumption}

\newtheorem*{definition*}{Definition}

\theoremstyle{remark}
\newtheorem{remark}[Def]{Remark}
\newtheorem{example}[Def]{Example}

\newtheorem*{claim*}{Claim}
\newtheorem*{remark*}{Remark}
\newtheorem*{example*}{Example}
\newtheorem*{notation*}{Notation}

\def\N{{\mathbb N}}

\def\R{{\mathbb R}}

\def\PP{{\mathbb P}}

\newcommand{\Ent}{\mathrm{Ent}}

\newcommand{\Cpl}{\mathrm{Cpl}}
\newcommand{\Pz}{\mathcal{P}}

\begin{document}

  \title[Heat Flow with Dirichlet Boundary Conditions and Gluing]{Heat Flow with Dirichlet Boundary Conditions via Optimal Transport and Gluing of Metric Measure Spaces}
  \author{Angelo Profeta, Karl-Theodor Sturm}
  \thanks{\emph{Acknowledgements:} Both authors gratefully acknowledge support by the German Research Foundation through the Excellence Cluster \emph{Hausdorff Center for Mathematics} and the CRC 1060  \emph{The Mathematics of Emergent Effects} as well as by the European Union through the ERC-AdG \emph{RicciBounds}. The first author would like to thank his mathematical sparring partners Susanne Hilger and Lorenzo Dello Schiavo.}
  \date{\today}
  \address{Institute for Applied Mathematics, University of Bonn, Endenicher Allee
	   60, 53115 Bonn, Germany}
  \email{profeta@iam.uni-bonn.de, sturm@uni-bonn.de}
  \begin{abstract}
  We introduce the \emph{transportation-annihilation distance} $W_p^\sharp$ between subprobabilities and derive contraction estimates with respect to this distance  for the heat flow with homogeneous Dirichlet boundary conditions on an open set in a metric measure space. 
  We also deduce the Bochner inequality for the Dirichlet Laplacian as well as gradient estimates for the associated Dirichlet heat flow.
  
  For the Dirichlet heat flow, moreover, we establish a gradient flow interpretation within a suitable  space of \emph{charged probabilities}.
  In order to prove this, we will work with the \emph{doubling} of the open set, the space obtained by gluing together two copies of it along the boundary. 
  \end{abstract}
   \maketitle

\allowdisplaybreaks

\section{Introduction and Statement of Main Results} \label{sec:introduction}
We present an approach to heat flow with homogeneous Dirichlet boundary conditions via optimal transport -- indeed, the very first ever --  based on a novel particle interpretation for this evolution.
The classical particle interpretation for the heat flow in an open set $Y$ with Dirichlet boundary condition is based on
particles which move around in $Y$ and are killed (or lose their mass) as soon as they hit the boundary $\partial Y$.
Our new interpretation will be based on particles moving around in $Y$, which are reflected if they hit the boundary, and which thereby randomly change their ``charge'': half of them change into ``antiparticles'', half of them continue to be normal particles. Effectively, they annihilate each other but the total number of charged particles remains constant.

This leads us to regard the initial probability distribution  as a distribution $\sigma_0^+$ of normal particles, with no antiparticles being around at time 0, i.e. $\sigma_0^-=0$. In the course of time, $\sigma_t^+$ and $\sigma_t^-$ will evolve as subprobability measures on $Y$ and so does the ``effective distribution'' $\sigma_t^0:=\sigma_t^+-\sigma_t^-$ whereas the ``total distribution'' $\overline\sigma_t :=\sigma_t^++\sigma_t^-$ continues to be a probability measure. The latter will evolve as heat flow with Neumann boundary conditions whereas the former will evolve as heat flow with Dirichlet boundary conditions. The evolution of the charged particle distribution $\sigma_t=(\sigma_t^+,\sigma_t^-)$ will be characterized as an EVI-gradient flow for the Boltzmann entropy. New transportation distances for \emph{subprobability} measures will yield contraction estimates for the effective flow.

Technically, we will interpret the pairs of subprobability measures $(\sigma^+,\sigma^-)$ as a \emph{probability} measure on the doubling of $Y$ in $X$, i.e.\! a space obtained by gluing together two copies of $X$ along the ''boundary`` $X\setminus Y$. Both settings are equivalent. Under a curvature condition for the doubling, we get Wasserstein contraction results and gradient estimates for the heat flow with Dirichlet boundary values. 

In particular, we also obtain the very first version of a  Bochner inequality for the  Dirichlet Laplacian on a convex subset of a Riemannian manifold -- which surprisingly involves both, the Dirichlet Laplacian and the Neumann Laplacian.

\subsection{Transportation-annihilation distance between subprobabilities}\label{subsec:Wnull}

Let $(X,d)$ be a complete separable metric space and $Y\subset X$ be an open subset with $\emptyset\not= Y\not= X$.
The distance between two normal particles at locations $x$ and $y\in X$ will be given by $d(x,y)$ --  and so is the distance between two antiparticles at $x$ and $y$. The distance between a normal particle at $x\in X$ and an antiparticle at $y\in X$ (or vice versa) will be given by
$$d^*(x,y):=\inf_{z\in X\setminus Y}\big[ d(x,z)+d(z,y)\big].$$
The set of {\em subprobability measures} on $Y$ (i.e.\! measures $\mu$ on $Y$ equipped with its Borel field with mass $\mu(Y)\le 1$) will be denoted by $\Pz^{sub}(Y)$.
Moreover, we introduce the set of {\em charged probability measures} on $X$ by
\begin{eqnarray*}
\tild\Pz(Y|X)&:=&\Big\{\sigma=(\sigma^+,\sigma^-) \,\big|\,  \sigma^\pm\in \Pz^{sub}(X), \ 
\sigma^+|_{X\setminus Y}=\sigma^-|_{X\setminus Y}, \ \sigma^+(X)+\sigma^-(X)=1\Big\}.
\end{eqnarray*}
The maps $\sigma\mapsto \sigma^0:=\sigma^+-\sigma^-$
and 
 $ \sigma\mapsto \overline\sigma:=\sigma^++\sigma^-$
 will assign the {\em effective measure} and the {\em total measure}, resp., to a charged probability measure. Observe that $\sigma^0$ is in general a \emph{signed} measure. However, we will mostly have charged measures with $\sigma^0\geq 0$ since we are usually starting with a subprobability $\mu$ and take an appropriate measure $\sigma$ such that $\sigma^0=\mu$. 
 
 Given $\sigma,\tau\in \tild\Pz(Y|X)$ and a coupling $q\in \Cpl(\overline\sigma,\overline\tau)$ of their total measures, there  are canonical decompositions
 $\sigma^i=\sigma^{i+}+\sigma^{i-}$,
 $\tau^j=\tau^{+j}+\tau^{-j}$,
 $q=q^{++}+q^{+-}+q^{-+}+q^{--}$ such that
 $q^{ij}\in\Cpl(\sigma^{ij},\tau^{ij})$ for $i,j\in\{+,-\}$.
 To construct these decompositions, choose nonnegative Borel functions $u^i, v^j$ on $X$ with $\sigma^i=u^i\, \overline \sigma$, $\tau^j=v^j\, \overline \tau$ and set
 $\dd q^{ij}(x,y):= u^i(x) v^j(y)\,\dd q(x,y)$
 as well as
 $\sigma^{ij}(\cdot):=q^{ij}(\cdot,X),\quad  \tau^{ij}(\cdot):=q^{ij}(X,\cdot)$.
 
 Having this canonical decomposition for $q\in \Cpl(\overline\sigma,\overline\tau)$ in mind, we define the \emph{$L^p$-transportation distance between charged probability measures $\sigma,\tau\in\tild\Pz(Y|X)$} by
 \begin{align}
    \tild W_p(\sigma,\tau) :=& \inf\Big\{ \int_{X\times X} d(x,y)^p \dd q^{++}(x,y)+\int_{X\times X} d^*(x,y)^p \dd q^{+-}(x,y) \notag\\
      &\qquad\quad+\int_{X\times X} d^*(x,y)^p \dd q^{-+}(x,y)+\int_{X\times X} d(x,y)^p \dd q^{--}(x,y) \,\Big|\, q\in \Cpl(\overline\sigma,\overline\tau)\Big\}^{1/p}  \label{eq:Wtilde}
\end{align}
for $p\in [1,\infty)$.

Define
$\tild\Pz_p(Y|X):=\big\{\sigma\in
\tild\Pz(Y|X)\,|\, \ \tild W_p\big(\sigma,(\frac12\delta_x,\frac12\delta_x)\big)<\infty$ for some/all $x\in X\big\}$.
Obviously, the map $\mu\mapsto\big(\frac12\mu,\frac12\mu\big)$ defines an isometric embedding of 
$\Pz_p(X)$ into $\tild\Pz_p(Y|X)$.

Based on an isometry between $\tild\Pz_p(Y|X)$ and $\Pz_p(\hat X)$ with a suitable ``glued space'' $\hat X$, we will deduce important metric properties of $\tild W_p$,
see Section \ref{subsec:identification}:
\begin{lemma} \label{lem:hatW_metric}
For each $p\in[1,\infty)$, $\tild W_p$ is a complete separable metric on $\tild\Pz_p(Y|X)$.
It is a length metric if $d$ is a length metric;  $\tild\Pz_p(Y|X)$ is compact if
 $X$ is compact.
\end{lemma}

Now we are in position to define the
$L^p$-transportation semi-metric between subprobabilities. 
\begin{Def} \label{def:Wnull}
For  $\mu,\nu\in\Pz^{sub}(Y)$ and $p\in[1,\infty)$ we define
 \begin{align}
 W^0_p(\mu,\nu):=& \inf\Big\{\tild W_p(\sigma,\tau) \,\Big|\, \sigma,\tau\in\tild\Pz(Y|X), \sigma^0=\mu, \tau^0=\nu\Big\}
 \label{eq:Wnull}
\\
=&\inf\Big\{\tild W_p\big( (\mu+\rho,\rho), (\nu+\eta,\eta)\big) \,\Big|\, \rho,\eta\in\Pz^{sub}(X), 
(\mu+2\rho)(X)=1, (\nu+2\eta)(X)=1\Big\}, \label{w0 vs what2} 
\end{align}
called the \emph{transportation-annihilation pre-distance}.
Moreover, we let
\[ \Pz^{sub}_p(Y):=\big\{ \mu\in\Pz^{sub}(Y) \,\big|\,  W^0_p(\mu,\delta_y)<\infty \text{ for some/all } y\in Y\big\}. \]
\end{Def}

\begin{remark}\label{remark}
	\begin{itemize} 
		\item[a)] The infima  in the previous Definition  will be attained if $X$ is compact. 
		\item[b)] If $\mu$ and $\nu$ are {\em probability} measures, then $ W^0_p(\mu,\nu)$ coincides with the usual $L^p$-Kanto\-rovich-Wasserstein metric $W_p(\mu,\nu)$.
		\item[c)] In general, $ W^0_p$ will not satisfy the triangle inequality. For instance, let $X=\R, Y=(-3,3), \mu=\delta_{-2}, \nu=\delta_2, \xi=0$. Then
			\[ W^0_p(\mu,\nu)=4\not\le W_p^0(\mu,\xi)+W^0_p(\xi,\nu)=2. \]
		\item[d)] The constraints $(\mu+2\rho)(X)=1, (\nu+2\eta)(X)=1$ can equally  well be replaced by the seemingly weaker constraints $(\mu+2\rho)(X)\leq 1,(\nu+2\eta)(X)\leq 1$. Indeed, whenever we have subprobabilities such that the constraints hold with ``$\leq 1$'', the finiteness of $\tilde W_p((\mu+\rho,\rho),(\nu+\eta,\eta))$ implies that $(\mu+2\rho)(X)=(\nu+2\eta)(X)$. But then we can choose an arbitrary subprobability $\vartheta$ with $\vartheta(X)=\frac{1}{2}(1-(\mu+2\rho)(X))$ and define $\tilde \rho:= \rho+\vartheta, \tilde\eta:=\eta+ \vartheta$. These subprobabilities now satisfy $(\mu+2\tilde\rho)(X)=1=(\nu+2\tilde\eta)(X)$ and we have
			\[ \tilde W_p((\mu+\tilde\rho,\tilde\rho),(\nu+\tilde\eta,\tilde\eta)) \leq \tilde W_p((\mu+\rho,\rho),(\nu+\eta,\eta)). \]
	\end{itemize} 
\end{remark}

%%%%%%%%%%%%%%

\medskip 

To overcome the lack of a triangle inequality for $W_p^0$, we now strive for a related length metric.
In a first step, we define a (pseudo-) metric, and out of this the induced \emph{length} (pseudo-) metric. 

\begin{Def} \label{def:W0_length}
\begin{itemize}
 \item[i)]  Given $\mu,\nu\in\mathcal P^{sub}_p(Y)$, let
  \begin{equation} W_p^\flat(\mu,\nu) := \inf \left\{ \sum_{i=1}^n W_p^0(\eta_{i-1},\eta_{i}) \,\Big|\, n\in\mn, \eta_i\in\mathcal P^{sub}_p(Y), \eta_0=\mu,\eta_n=\nu \right\}. \label{eq:Wflat}
  \end{equation}
 \item[ii)] Given a curve $(\eta_{s})_{s\in [0,1]}\subset \mathcal P^{sub}_p(Y)$, we define its $W_p^\flat$-length by
  \[ L_p^\flat(\eta) := \sup \left\{ \sum_{i=1}^n W_p^\flat(\eta_{s_{i-1}}, \eta_{s_i}) \,\Big|\, n\in\mn, 0=s_0<\ldots<s_n=1 \right\}. \]
  \item[iii)] For two measures $\mu,\nu\in\mathcal P^{sub}_p(Y)$, the induced length metric is now obtained by
    \begin{equation} 
    W_p^\sharp (\mu,\nu) := \inf \left\{ L_p^\flat(\eta) \,\big|\, \eta\colon [0,1] \to \mathcal P^{sub}_p(Y)\, W_p^\flat\text{-continuous, } \eta_0=\mu,\eta_1=\nu \right\}. \label{eq:Wsharp}
    \end{equation}
    It will be called \emph{transportation-annihilation distance}.
\end{itemize}
\end{Def}

\begin{remark}
Both, $W_p^\flat$ and $W_p^\sharp$ are a priori only pseudo-metrics; the former the biggest one below $W_p^0$, the latter the smallest \emph{intrinsic} one above $W_p^\flat$. In what follows, it will turn out however that  both indeed are metrics and for $p=1$ they coincide.
\end{remark}

\medskip

We will compare the previous (pseudo-)metrics with the Kantorovich-Wasserstein metric $W'_p$ on  the one-point completion $(Y',d')$ of $Y$. Here $Y':=Y\cup \{\partial\}$ and the \emph{shortcut metric} $d'$ is given by 
  \begin{equation} 
    d'(x,y) := \min \{ d(x,y), d'(x,\partial) + d'(y,\partial) \}, \label{eq:Wprime}
  \end{equation}
for $x,y\in Y$,
 $d'(x,\partial) = d'(\partial,x) := \inf_{z\in X\setminus Y} d(x,z)$, and  $d'(\partial,\partial):=0$.
If $(X,d)$ is a complete, length metric space then so will be  $(Y',d')$ .
  If in addition $X$ is proper (i.e.\ closed balls are compact) then  $(Y',d')$  will be a geodesic space.
  
  We will further denote $d^\dagger(x,y):= d'(x,\partial) + d'(y,\partial)$, so that $d'=\min\{ d,d^\dagger \}$.

\begin{Def}
 \begin{itemize}
 	\item[i)] $W'_p$ will denote the $L^p$-Kantorovich-Wasserstein distance on $\mathcal P_p(Y')$ induced by the distance $d'$.
 	\item[ii)] Extending each subprobability measure $\mu\in \mathcal P^{sub}(Y)$ to a probability measure $\mu'\in \mathcal P(Y')$ by
 	$\mu':=\mu + (1-\mu(Y))\delta_\partial$ induces a bijective embedding of 
 	$\mathcal P^{sub}(Y)$  into $\mathcal P(Y')$.
 	The induced distance on $\mathcal P^{sub}(Y)$ will  again  be denoted by $W'_p$.
 	\item[iii)] For subprobability measures $\mu,\nu$ \emph{of equal mass} we will also make use of the transportation cost 
 		\begin{equation} 
 		W_p^\dagger(\mu,\nu)^p := \inf_{q \in \operatorname{Cpl}(\mu,\nu)} \int_{Y\times Y} d^\dagger(x,y)^p \dd q(x,y) \label{eq:Wdagger}
 		\end{equation}
 		induced by $d^\dagger$. 
 	\item[iv)] Finally, for subprobabilities of equal mass define the $L^p$-transportation distance with respect to the meta-metric $d^*$
 	\begin{equation} 
 	W_p^*(\mu,\nu)^p := \inf_{q\in \Cpl(\mu,\nu)} \int_{X\times X} d^*(x,y)^p \dd q(x,y), \label{eq:Wstar}
 	\end{equation}
 	and let $W_p^*(\mu):=\frac12W_p^*(\mu,\mu)$, which will be called \emph{annihilation cost of the subprobability $\mu$}. 
 \end{itemize}
\end{Def}

\begin{remark}
	Obviously, $W_p^*$ is symmetric in its arguments and satisfies the triangle inequality but typically $W_p^*(\mu,\mu) \not=0$.
\end{remark}

\begin{example} 
	Let $X=\R, Y=(-1,1), \mu=\delta_{x}, \nu=\delta_y$ for $x,y\in Y$. Then
	\[ W^0_p(\mu,\nu)=|x-y|,\qquad W^\flat_p(\mu,\nu)=W^\sharp_p(\mu,\nu)=W_p'(\mu,\nu)=\min\{|x-y|,2-|x-y|\}.\]
\end{example}

\begin{remark} \label{rem:W'_FigalliGigli}
 One could equally well define 
 	\[ W_p''(\mu,\nu) := \inf \{ W_p(\check{\mu},\check{\nu}) \,\big|\, \check{\mu},\check{\nu}\in \mathcal M(Y'), \check{\mu}|_Y=\mu, \check{\nu}|_Y=\nu \}. \]
 For $p=1$ the metrics $W_1'$ and $W_1''$ coincide, but for $p>1$ this is no longer true. Take for instance $X=\mr$, $Y=(-3,3)$ and $\mu=\delta_{-2},\nu=\delta_2$. Then 
 $W_p'(\mu,\nu)^p=d'(-2,2)^p=2^p$ whereas $W_p''(\mu,\nu)^p
 \le d'(-2,\partial)^p+d'(2,\partial)^p=2$. \newline
 The metric $W_2''$ coincides with Figalli \& Gigli's metric $Wb_2$ \cite{Figalli_Gigli}.
\end{remark}

%%%%%%%%%%%%%%
From now on until the end of this subsection  assume that {\bf $(X,d)$ is a {length} space.} 

\medskip

Quite intuitive characterizations of $W^0_p(\mu,\nu)$, $W^\sharp_p(\mu,\nu)$, and $W'_p(\mu,\nu)$ are possible   in terms of  $L^p$-transportation costs and and  $L^p$-annihilation costs.

\begin{lemma} \label{lem:rep_W0_1}
i) For all $\mu,\nu\in\mathcal P^{sub}_1(Y)$
  \begin{align*} 
    {W^0_1}(\mu,\nu) = \inf \Big\{ & W_1(\mu_1,\nu_1) + W_1^*(\mu_0) + W^*_1(\nu_0) \,\Big|\, \\
      & \mu=\mu_1+\mu_0, \nu=\nu_1+\nu_0, (\mu+\nu_0)(X)\leq 1, (\nu+\mu_0)(X)\leq 1 \Big\}.
  \end{align*}
ii) More generally for all $p\ge1$ and $\mu,\nu\in\mathcal P^{sub}_p(Y)$
  \begin{align*} 
    {W^0_p}(\mu,\nu)^p \le \inf \Big\{ & W_p(\mu_1,\nu_1)^p + W_p^*(\mu_0)^p + W^*_p(\nu_0)^p \,\Big|\, \\
      & \mu=\mu_1+\mu_0, \nu=\nu_1+\nu_0, (\mu+\nu_0)(X)\leq 1, (\nu+\mu_0)(X)\leq 1 \Big\}.
  \end{align*}
iii) For all $\mu,\nu\in\mathcal P^{sub}_1(Y)$
  \begin{align*} 
    W_1^\sharp(\mu,\nu) = \inf \Big\{ & W_1(\mu_1,\nu_1) + W_1^*(\mu_0) + W^*_1(\nu_0) \,\Big|\,        \mu=\mu_1+\mu_0, \nu=\nu_1+\nu_0 \Big\}.
  \end{align*}
iv) For all $\mu,\nu\in\mathcal P^{sub}_p(Y)$
  \begin{align} \nonumber
    W'_p(\mu,\nu)^p &= \inf \Big\{  W_p(\mu_1,\nu_1)^p + W^\dagger_p(\mu_2,\nu_2)^p+ W'_p(\mu_0,0)^p + W'_p(\nu_0,0)^p \,\Big|\\
    &\qquad\qquad  \mu=\mu_1+\mu_2+\mu_0, \ \nu=\nu_1+\nu_2+\nu_0,\  (\mu+\nu_0)(Y)\le1, \ (\nu+\mu_0)(Y)\le1\Big\}
    \label{W'-decompo}
  \end{align}
  where  
  $W'_p(\mu_0,0)^p = \int_Y d'(x,\partial)^p \dd \mu_0(x)$ with 0 denoting the subprobability measure with vanishing total mass.

In the case $p=1$, contributions from the term $W^\dagger_p(\mu_2,\nu_2)^p$ can be avoided, in other words, one can always choose $\mu_2=\nu_2=0$.
\end{lemma}

\begin{lemma} \label{lem:Wstern_rand}
For all $p\ge 1$ and all $\mu\in \mathcal P_p(Y)$
$$2^{-1+1/p}\, W'_p(\mu,0)\le W^*_p(\mu)\le W'_p(\mu,0)=\inf \big\{ W_p(\mu,\xi) \,\big|\, \ \xi\in \mathcal P( \partial Y)\big\}.$$

In particular,
$W^*_1(\mu)=W'_1(\mu,0)$. More generally,
for all $\mu,\nu\in \mathcal P_1(Y)$
$$W^*_1(\mu,\nu)=\inf \Big\{ W_1(\mu,\xi)+W_1(\xi,\nu) \,\big|\, \ \xi\in \mathcal P( \partial Y)\Big\}.$$
\end{lemma}

\begin{remark} In general, $W^*_p(\mu)$ and $W'_p(\mu,0)$ will not coincide. Our lower bound for $W^*_p(\mu)/W'_p(\mu,0)$ is sharp.

For instance, let $Y=(0,2)\subset X=\R$ and $\mu=\frac12(\delta_1+\delta_\eps)$ for some $\eps\in(0,1)$. Then $W'_p(\mu,0)^p=\frac12(1+\eps^p)$ whereas $W^*_p(\mu)^p=\big(\frac{1+\eps}2\big)^p$. Thus for $\eps$ sufficiently small,
$W^*_p(\mu)/W'_p(\mu,0)$ is arbitrarily close to $2^{-1+1/p}$.
\end{remark}
%\medskip

\begin{theorem} \label{thm:W_sharp-W_prime}
i)
 For all $\mu,\nu\in\mathcal P^{sub}_1(Y)$
$$ W_1^\flat(\mu,\nu) = W_1^\sharp(\mu, \nu)=W_1'(\mu, \nu).$$

ii) More generally, for all $p\ge1$ and  all $\mu,\nu\in\mathcal P^{sub}_p(Y)$
$$W_1'(\mu, \nu)\le W_p^\flat(\mu,\nu)\le
W^\sharp_p(\mu,\nu)\le W'_p(\mu,\nu).$$
\end{theorem}

\begin{example} 
	Let $X=\R, Y=(-2,2), \mu=\frac1{2n+1}\delta_{-1/2}, \nu=\frac1{2n+1}\delta_{+1/2}$ for $n\in \N$. Then
	\[ W'_p(\mu,\nu)^p=W_p(\mu,\nu)^p=\frac1{2n+1}. \]
	Taking 
	\[ \sigma:= \left(	\frac{1}{2n+1} \sum_{k=0}^n \delta_{\frac{2k}{2n+1}-\frac{1}{2}}, \frac{1}{2n+1} \sum_{k=1}^n \delta_{\frac{2k}{2n+1}-\frac{1}{2}} \right) \] 
	and 
	\[ \tau:= \left( \frac{1}{2n+1} \sum_{k=0}^{n} \delta_{\frac{2k+1}{2n+1}-\frac{1}{2}}, \frac{1}{2n+1} \sum_{k=0}^{n-1} \delta_{\frac{2k+1}{2n+1}-\frac{1}{2}} \right), \] 
	we see that
	\[ W_p^0(\mu,\nu)^p \leq \tilde W_p (\sigma,\tau)^p = \left( \frac{1}{2n+1} \right)^p, \]
	so that 
	\[ W_p^\flat(\mu,\nu) \leq W_p^0(\mu,\nu) \leq \left( \frac{1}{2n+1} \right) < \left( \frac{1}{2n+1} \right)^{\frac{1}{p}} = W_p'(\mu,\nu), \]
	for $p>1$, $n\geq 1$.
	In particular, the lower estimate for $W^\flat_p$ in assertion ii) of the previous Theorem is sharp.
\end{example}

%\begin{example} Let $X=\R, Y=(-2,2), \mu=\frac1{2n+1}\delta_{-1/2}, \nu=\frac1{2n+1}\delta_{+1/2}$ for $n\in \N$. Then
%$$W'_p(\mu,\nu)^p=W_p(\mu,\nu)^p=\frac1{2n+1}$$
%whereas
%\begin{eqnarray*}W^\sharp_p(\mu,\nu)^p&\le& W^0_p(\mu,\nu)^p\le \tilde W_p\Big(\big(
%\frac1{2n+1}\sum_{k=0}^n \delta_{\frac{2k}{2n+1}-\frac12}, \frac1{2n+1}\sum_{k=1}^n \delta_{\frac{2k}{2n+1}-\frac12}
%\big),\\
%&&\qquad\qquad\qquad\qquad\big(
%\frac1{2n+1}\sum_{k=0}^{n} \delta_{\frac{2k+1}{2n+1}-\frac12}, \frac1{2n+1}\sum_{k=0}^{n-1} \delta_{\frac{2k+1}{2n+1}-\frac12}
%\big)\Big)^p
%=\Big(\frac1{2n+1}\Big)^p.
%\end{eqnarray*}
%Thus in this example
%$$W^\sharp_p(\mu,\nu)\leq W'_1(\mu,\nu)=\Big(\frac1{2n+1}\Big)\ll \Big(\frac1{2n+1}\Big)^{1/p}= W'_p(\mu,\nu)$$
%for $p>1$ and $n\gg 1$.
%In particular,
%the lower estimate for $W^\sharp_p$ in assertion ii) of the previous Theorem is sharp.
%\end{example}

A useful feature of $W_p^\sharp$ is that it metrizes vague convergence of subprobability measures.
\begin{proposition} \label{cor:W_sharp_vague_conv}
Assume that $X$ is a compact geodesic space.
Then for every $p\in[1,\infty)$, 
 $W^\sharp_p$ is a complete, separable, geodesic metric on  $\mathcal P^{sub}_p(Y)$ and for $\mu_n,\mu \in \mathcal P^{sub}_p(Y)$ the following are equivalent:
 \begin{itemize}
 \item[(i)] $\mu_n\to\mu$ vaguely on $Y$.
 \item[(ii)] $W^\sharp_p(\mu_n,\mu)\to 0$ as $n\to\infty$ 
 \end{itemize} 
\end{proposition}

\begin{remark} 
In particular, this implies that $\mu_n\to\mu$ weakly on $Y$ if and only if $W^\sharp_p(\mu_n,\mu)\to 0$ and $\mu_n(Y)\to \mu(Y)$.
A similar result for  $W^0_p$ can be deduced even without requiring that $X$ is geodesic, see Lemma \ref{lem:W_sharp_weak_conv}.

The implication ``(ii)$\Rightarrow$(i)''  holds true for all length spaces $X$ without requiring their compactness.
For the converse, one has to add a condition on convergence of moments, see remark following Lemma \ref{lem:W_sharp_weak_conv}.
\end{remark}

\subsection{Gradient flow perspective and transportation estimates} \label{subsec:gradient_flow_perspective}
From now on, let us be more specific. We assume that
$(X,d,\mm)$ is a metric measure space which satisfies an RCD$(K,\infty)$-condition for some number  $K\in\R$ and that $Y\subset X$ is a dense open subset with $\mm(\partial Y)=0$.  The RCD$(K,\infty)$-condition means that the metric measure space $(X,d,\mm)$  is infinitesimally Hilbertian with Ricci curvature bounded from below by $K$ in the sense of Lott-Sturm-Villani, \cite{STU06i}, \cite{LV09}.  The latter is formulated as $K$-convexity of the Boltzmann entropy $\Ent_\mm$ in $\big(\Pz_2(X), W_2\big)$.  We will additionally request that this property extends to the space of charged probability measures induced by $Y$, that is, we will request that  $(X,Y,d,\mm)$ satisfies the following:

\begin{assumption}[``Charged Lower Ricci Bound $K$''] \label{ass} The Boltzmann entropy
 \begin{eqnarray*}
 \widetilde\Ent_\mm: \quad \tild\Pz_2(Y|X)&\to& (-\infty,\infty]\\
 \sigma&\mapsto& \Ent_\mm(\sigma^+)+\Ent_\mm(\sigma^-)
 \end{eqnarray*}
 is $K$-convex in the metric space $\big(\tild\Pz_2(Y|X), \tild W_2\big)$. 
\end{assumption}

\begin{remark} 
a) Note that, due to the isometric embedding of $\Pz_2(X)$ into $\tild\Pz_2(Y|X)$, this assumption will imply the $K$-convexity of  $\Ent_\mm$ in $\big(\Pz_2(X), W_2\big)$ and thus the CD$(K,\infty)$-condition for the metric measure space $(X,d,\mm)$.

b) If $(X,d,\mm)$  is infinitesimally Hilbertian and if $\mm$ has full topological support then Assumption \ref{ass} implies that $\overline Y=X$.
Indeed, the argument from \cite{Rajala_Sturm} carries over to this framework and yields essential non-branching which in turn implies the density of $Y$ in $X$.
\end{remark}

The \emph{proofs} of the following results will be given in Section \ref{sec:proofs}. They will be based on concepts and results for gluing of metric measure spaces which will be presented in Section \ref{sec:gluing}. For the various kinds of heat flows appearing from this section on, see Subsection \ref{subsec:heat_flows}.

\begin{theorem}\label{riem} 
 Let $(M,g)$ be a complete Riemannian manifold with Ricci curvature bounded below by $K\in \mr$. Take an open, bounded, convex subset $Y\subset M$ with smooth, compact boundary. Consider the closure $X:= \overline Y$ with the Riemannian distance $d$ and the Riemannian volume measure $\mm$ obtained by restriction to $X$. Then the metric measure space $(X,d,\mm)$ satisfies the RCD$(K,\infty)$-condition and $(X,Y,d,\mm)$ satisfies Assumption \ref{ass}.
\end{theorem}

 \begin{proposition}\label{evi}  
 Assume that Assumption \ref{ass} holds.
 \begin{itemize}
 \item[i)] For each $\sigma_0\in\tild\Pz_2(Y|X)$, there exists a unique $\evik$-gradient flow $(\sigma_t)_{t>0}$  for the Boltzmann entropy $\widetilde\Ent_\mm$ in $\big(\tild\Pz_2(Y|X), \tild W_2\big)$.

 \item[ii)] For each $\mu_0\in\Pz_2^{sub}(Y)$,
 the heat flow $(\mu_t)_{t>0}$ on $Y$ with Dirichlet boundary conditions is obtained as the effective flow 
 $$\mu_t=\sigma^+_t-\sigma^-_t$$
 where $(\sigma_t)_{t>0}$ is the $\evik$-flow as above starting in any $\sigma_0\in\tild\Pz_2(Y|X)$ with $\mu_0=\sigma^+_0-\sigma^-_0$.
 
 \item[iii)]
 For each $\nu_0\in\Pz_2(X)$,
 the heat flow $(\nu_t)_{t>0}$ on $X$ %with Neumann boundary conditions 
 is obtained as the total flow 
 $$\nu_t=\sigma^+_t+\sigma^-_t$$
 where $(\sigma_t)_{t>0}$ is the $\evik$-flow as above starting in 
 any $\sigma_0\in\tild\Pz_2(Y|X)$ with $\nu_0=\sigma^+_0+\sigma^-_0$.
  
 \item[iv)] For each $\sigma_0\in\tild\Pz_2(Y|X)$, the $\evik$-flow $(\sigma_t)_{t>0}$ from i) can be characterized as
 $$\sigma_t=\Big( \frac{\nu_t+\mu_t}2,  \frac{\nu_t-\mu_t}2\Big)$$
 where $(\nu_t)_{t>0}$ will denote the heat flow on $X$ %with Neumann boundary conditions 
 starting in $\nu_0=\sigma^+_0+\sigma^-_0$ and 
 $(\mu_t)_{t>0}$ will denote the heat flow on $Y$ with Dirichlet boundary conditions starting in $\mu_0=\sigma^+_0-\sigma^-_0$.
 \end{itemize}
 \end{proposition}

 \begin{remark}
  
 a) As in \cite{Savare}*{after Cor. 4.3, Thm. 4.4} (based on \cite{AGS-Bakry-Emery}*{Prop. 3.2, Thm. 3.5}) one can extend the flow to measures without finite second moment.
 
 b) In the situation of Theorem \ref{riem}, the ``heat flow on $X$'' will be the heat flow on $\overline Y\subset M$ with Neumann boundary conditions at $\partial Y$.
 \end{remark}

\begin{proposition}\label{prop:W_contraction} 
The $\evik$-flows $(\sigma_t)_{t>0}$ and $(\tau_t)_{t>0}$ as above are $K$-contractive in all $L^p$-transportation distances:
$$\tild W_p\big( \sigma_t,\tau_t)\le e^{-Kt}\cdot \tild W_p\big( \sigma_0,\tau_0)$$
for all $t>0$ and all $p\in[1,\infty)$.
 \end{proposition}

\begin{theorem} \label{thm:Wnull_contraction}
For all $\mu_0,\nu_0\in\Pz^{sub}_p(Y)$, all $t>0$ and all $p\in[1,\infty)$
$$ W^0_p\big( \mu_t,\nu_t)\le e^{-Kt}\cdot  W^0_p\big( \mu_0,\nu_0)$$
where  $(\mu_t)_{t>0}$ and $(\nu_t)_{t>0}$ denote the heat flows on $Y$ with Dirichlet boundary conditions starting in $\mu_0$ and $\nu_0$, resp.
\end{theorem}

\begin{proof}
Given $\mu_0,\nu_0\in\Pz^{sub}_p(Y)$ and $\eps>0$,  we may choose $\sigma_0,\tau_0\in \tilde\Pz_p(Y|X)$ with 
$\mu_0=\sigma^+_0-\sigma^-_0$ and  $\nu_0=\tau^+_0-\tau^-_0$ such that
$$\tild W_p\big( \sigma_0,\tau_0)\le W^0_p\big( \mu_0,\nu_0)+\eps.$$
Thus, by the very definition of  $ W^0_p$ and by the previous proposition,
\begin{eqnarray*}
 W^0_p\big( \mu_t,\nu_t) \le 
\tild W_p\big( \sigma_t,\tau_t)
\le  e^{-Kt}\cdot \tild W_p\big( \sigma_0,\tau_0)=e^{-Kt}\cdot \Big( W^0_p\big( \mu_0,\nu_0)+\eps\Big).
\end{eqnarray*}
Since $\eps>0$ was arbitrary, this proves the claim.
\end{proof}

\begin{corollary} \label{cor:Wsharp_contraction}
Let $\mu_0,\nu_0\in\Pz^{sub}_p(Y)$, and $(\mu_t)_{t>0}$ and $(\nu_t)_{t>0}$ denote the heat flows on $Y$ with Dirichlet boundary conditions starting in $\mu_0$ and $\nu_0$, resp. Then for all $t>0$ and all $p\in[1,\infty)$ we have both
 \[ W^\flat_p\big( \mu_t,\nu_t)\leq e^{-Kt}\cdot  W^\flat_p\big( \mu_0,\nu_0), \]
and
  \[ W^\sharp_p\big( \mu_t,\nu_t)\le e^{-Kt}\cdot  W^\sharp_p\big( \mu_0,\nu_0). \]
  In particular,
 $W'_1\big( \mu_t,\nu_t)\le e^{-Kt}\cdot  W'_1\big( \mu_0,\nu_0)$.
\end{corollary}

\begin{proof}
Observe that
 \begin{align*}
  W_p^\flat(\mu_t,\nu_t) =& \inf \left\{ \sum_{i=1}^n W_p^0(\eta_{i-1}, \eta_i) \,\big|\, n\in\mn, \eta_i\in\mathcal P^{sub}_p(Y), \eta_0=\mu_t, \eta_n=\nu_t \right\} \\
    \leq& \inf \left\{ \sum_{i=1}^n W_p^0(\mathscr P_t^0\xi_{i-1}, \mathscr P_t^0\xi_i) \,\big|\, n\in\mn, \xi_i\in\mathcal P^{sub}_p(Y), \xi_0=\mu_0, \xi_n=\nu_0 \right\} \\
    \leq& e^{-Kt} \inf \left\{ \sum_{i=1}^n W_p^0(\xi_{i-1}, \xi_i) \,\big|\, n\in\mn, \xi_i\in\mathcal P^{sub}_p(Y), \xi_0=\mu_0, \xi_n=\nu_0 \right\} \\
    =& e^{-Kt} W_p^\flat(\mu_0,\nu_0).
 \end{align*}
 Here, $\mathscr P_t^0$ is the heat semigroup with Dirichlet boundary conditions on measures, see Subsection \ref{subsec:heat_flows}. 
 This also implies that for a curve $(\eta_s)_{s\in[0,1]}\subset \mathcal P^{sub}_p(Y)$ its length satisfies $L_p^\flat(\mathscr P_t\eta) \leq e^{-Kt} L_p^\flat(\eta)$, so that eventually
 \begin{align*}
  W_p^\sharp(\mu_t,\nu_t) =\inf_{\eta:\mu_t\leadsto \nu_t} L_p^\flat(\eta) \leq \inf_{\xi:\mu_0 \leadsto \nu_0} L_p^\flat(\mathscr P_t \xi) \leq e^{-Kt} \inf_{\xi:\mu_0 \leadsto \nu_0} L_p^\flat(\xi) = e^{-Kt} W_p^\sharp(\mu_0,\nu_0).
 \end{align*}
\end{proof}

\subsection{Gradient estimates and Bochner's inequality} \label{subsec:gradient_estimate_bochner}
 Let us continue to assume that $(X,d,\mm)$ is a metric measure space which satisfies an RCD$(K,\infty)$-condition  and that $Y\subset X$ is a dense open subset with $\mm(\partial Y)=0$. Assumption \ref{ass} yields a gradient estimate which involves both semigroups, $P_t$ (with Neumann boundary condition) and $P_t^0$ (with Dirichlet boundary condition). Before proving this estimate, we will see that it is equivalent to a Bochner inequality which involves the corresponding Laplace operators. To state directly the $p$-versions, let us introduce the appropriate function spaces. For $p\in [1,\infty)$ we set
  \begin{align}
   D_p(\mathcal E) :=& \{ f\in D(\mathcal E)\cap L^p(X,\mm) \,\big|\, |\nabla f| \in L^p(X,\mm) \}, \label{eq:def_D_p} \\
   D_p(\Delta) :=& \{ f\in D(\Delta) \cap L^p(X,\mm) \,\big|\, \Delta f\in L^p(X,\mm) \}, \label{eq:def_D_p_Laplace}
  \end{align}
 and similarly for $\mathcal E^0$ and $\Delta^0$, which are the Dirichlet form and generator associated to the heat flow $P_t^0$.

\begin{proposition}\label{prop:p_grad=p_bochner} 
Assume that $\mm(X)<\infty$. For each $p\in [1,2]$, the following 
properties are equivalent to each other:
\begin{itemize}
\item[(i)] 
For all $t>0$, and all $f\in D_p(\mathcal E^0)$
\begin{equation} \big|\nabla P^0_tf\big|^p\le e^{-Kpt}\cdot P_t\big(|\nabla f|^p\big) \;\;\;\;\; \text{$\mm$-a.e.\! in $X$} \;\;\;\text{ (``$p$-gradient estimate'')}. \label{eq:p_grad_est}\end{equation}
Note that  different semigroups appear on the left and right hand side.
\item[(ii)] For all $f\in D_p(\Delta^0)$ with $\Delta^0f\in D_p(\mathcal E^0)$ and every $\varphi\in D_\infty(\Delta)$ with $\varphi\geq 0$ 
  \begin{equation} \frac{1}{p} \int_X \Delta\varphi |\nabla f|^p \dm - \int_{\{|\nabla f| \neq 0 \}} \varphi |\nabla f|^{p-2} \nabla f \cdot \nabla \Delta^0 f\dm \geq K \int_X \varphi |\nabla f|^p \dm \;\;\; \text{ (``$p$-Bochner inequality'')}. \label{eq:p_bochner} \end{equation}
\end{itemize}
\end{proposition}

The proof is an adaption of the one of \cite{Han}*{Thm. 3.6}.

\begin{theorem} \label{thm:ass_implies_bochner}
\begin{itemize}
\item[i)] Assumption \ref{ass} implies that both properties (i) and (ii) of Proposition \ref{prop:p_grad=p_bochner} are satisfied, even for all  $p\in [1,\infty)$ and without the assumption that $\mm(X)<\infty$.
\item[ii)] Moreover, it implies that the flows from Proposition \ref{evi} and the heat semigroups are related to each other by
$$\nu_t=(P_tv)\mm, \qquad \mu_t=(P_t^0w)\mm$$
for $\nu_0=v\mm \in \mathcal P_2(X)$ and $\mu_0=w\mm \in \mathcal P^{sub}_2(X)$.
\end{itemize}
\end{theorem}

\begin{corollary} Assume \ref{ass}. Then for all $u:X\to\R$ and all $t>0$
$$\Lip_d(P_t^0u)\le e^{-Kt}\, \Lip_d(u)$$
as well as
$$\Lip_{d'}(P_t^0u)\le e^{-Kt}\, \Lip_{d'}(u).$$
Here $\Lip_{d}(.)$ denotes the Lipschitz constant w.r.t.\ the original metric $d$ on $X=\overline Y$ whereas $\Lip_{d'}(.)$ denotes the Lipschitz constant w.r.t.\ the shortcut metric $d'$ on $Y'=Y\cup\{\partial\}$.
\end{corollary}

\begin{proof} The $\Lip_{d}$-estimate follows from the previous gradient estimates \eqref{eq:p_grad_est} by taking supremum norm.
The $\Lip_{d'}$-estimate, on the other hand, follows via Kuwada duality from the transport estimate in Corollary \ref{cor:Wsharp_contraction} with $p=1$.
\end{proof}
\medskip

Let us finally give a geometric  characterization of Assumption \ref{ass}. Given a metric measure space $(V,d_V,\mm_V)$ we say that an open subset $U\subset V$ is a {\em halfspace} if there exists a measure-preserving isometry $\psi: V\to V$ with invariant set $\partial U=\{x\in V: \psi(x)=x\}$ such that
$\psi(U)=V\setminus \overline U$. We call two metric measure spaces $(V,d_V,\mm_V)$ and $(W,d_W,\mm_W)$ \emph{mms-isomorphic} if there exists a measure-preserving isometry $\xi\colon (V,d_V,\mm_V) \to (W,d_W,\mm_W)$.

\begin{theorem} \label{thm:halfspace}
Let $(X,d,\mm)$ be a metric measure space, and $Y\subset X$ an open local $\rcd{K}{\infty}$ space. 
The following properties are equivalent
\begin{itemize}
\item[(i)] Assumption $\ref{ass}$.
\item[(ii)] $Y$ is a halfspace  in some $\rcd{K}{\infty}$-space $(V,d_V,\mm_V)$ in the sense that there is a halfspace $\tilde Y\subset V$ and a measure-preserving isometry $\xi\colon (Y,d,\mm|_Y)\to (\tilde Y,d_V,\mm_V|_{\tilde Y})$.
\item[(iii)] $\partial Y$ is covered by open sets $X_i$ such that $Y\cap X_i$ for each $i$ is mms-isomorphic to a halfspace $W_i$ in some $\rcd{K}{\infty}$-space $(V_i,d_i,\mm_i)$.
\end{itemize}
\end{theorem}

\begin{remark} 
The heat flow with Dirichlet boundary values from an optimal transport perspective, to  our knowledge has so far only been investigated in \cite{Figalli_Gigli}, where the authors define a  transportation distance between measures allowing to create or destroy mass at the boundary. This metric is a modification of our transportation metric $W_2'$ based on the shortcut metric $d'$, see Remark \ref{rem:W'_FigalliGigli}. This leads to a gradient flow description of the heat equation with \emph{strictly positive, constant} Dirichlet boundary conditions. However, it does not apply to the study of the heat flow with vanishing Dirichlet boundary conditions. Further approaches to metrics on the space of finite Radon measures are given in \cites{LieroMielkeSavare, PiccoliRossi, KondratyevMV}.
\end{remark}
\medskip

\textbf{Structure of the paper:} In Section \ref{sec:introduction} we introduced the setting of particles and antiparticles, giving definitions, stating the main results and giving proofs of those results which do not need the doubling. 
Section \ref{sec:heat_flows_X} deals with the heat flow on metric measure spaces. In particular, the heat flow with Dirichlet boundary values is discussed.
In Section \ref{sec:gluing}, gluing of metric measure spaces is introduced and the space
of charged probability measures is identified with the space of probability measures on the doubled space.
Section \ref{sec:Wnull} is devoted to the detailed study of various (generalized) metrics on the space of probability measures.
Finally, in Section \ref{sec:proofs}, we present the remaining proofs of the results of Subsections \ref{subsec:gradient_flow_perspective} \& \ref{subsec:gradient_estimate_bochner}.

\medskip

In the sequel, the notion of a metric on a space $X$ will be crucial: it is a real-valued, symmetric function on $X\times X$ which satisfies the triangle inequality, vanishes  on the diagonal and is positive otherwise. We will also use several extensions which satisfy all but one of the above properties:
\begin{itemize}
\item extended metric: also the value $+\infty$ is admitted
\item pseudo-metric: may vanish also outside the diagonal
\item meta-metric: not necessarily vanishing on the diagonal
\item semi-metric: triangle inequality is not requested.
\end{itemize}
As we will encounter as much as 9 generalized ``$W$- metrics'', let us give a short overview where to find the definitions:
\medskip 

\noindent
\begin{itemize}
	\item[-] $W_p$ usual Kantorovich-Wasserstein metric on $\mathcal P_p(X)$
	\item[-] $\tilde W_p$ transportation metric on $\tilde{\mathcal P}_p(Y|X)$, \eqref{eq:Wtilde}
	\item[-] $W_p^0$ transportation-annihilation pre-metric on $\mathcal P^{sub}_p(Y)$, \eqref{eq:Wnull}
	\item[-] $W_p^\flat$ pseudo-metric on $\mathcal P^{sub}_p(Y)$, \eqref{eq:Wflat}
	\item[-] $W_p^\sharp$ transportation-annihilation metric on $\mathcal P^{sub}_p(Y)$, \eqref{eq:Wsharp}
	\item[-] $W_p'$ Kantorovich-Wasserstein metric on $\mathcal P_p(Y')$, based on shortcut metric $d'$, \eqref{eq:Wprime}
	\item[-] $W_p^\dagger$ transportation cost ``over the boundary'' on measures on $Y$ of the same mass, \eqref{eq:Wdagger} 
	\item[-] $W_p^*$ annihilation cost; meta-metric on measures on $X$ of the same mass, \eqref{eq:Wstar}
	\item[-] $\hat{W}_p$ Kantorovich-Wasserstein metric on $\mathcal P_p(\hat X)$, Lemma \ref{lem:phipsi}
\end{itemize}

\section{Metric measure spaces and heat flows} \label{sec:heat_flows_X}

\subsection{Gradients and Dirichlet forms} \label{subsec:dirichlet_forms} 
In this subsection we will introduce some notation and collect some results for Dirichlet forms on the original space $X$. 

Let $(X,d)$ be a complete, separable, length metric space, and let $\mm$ be a Borel measure with full support $\operatorname{supp}\mm=X$, satisfying the exponential integrability condition
\begin{equation} \int_X e^{-cd(x,x^*)^2} \dm(x) < \infty  \label{eq:expint}\end{equation}
for some $c>0,\,x^*\in X$.

The Cheeger energy of a function $f\in L^2(X,\mm)$ is defined as
  \[ \ch(f):= \inf \left\{ \liminf_{k\to \infty} \frac{1}{2}\!\int_{X}\! |\li(f_k)|^2 \dm  \,\Big|\, f_k\!\in\lip(X,d), \text{ s.t. } f_k\to f \text{ in } L^2(X,\mm) \right\}, \]
with domain $\mathcal F:=\{ f\in L^2(X,\mm) \,\big|\, \ch(f)<\infty\}$ (sometimes also denoted by $D(\ch)$ or $\sob(X,d,\mm)$). 
Here $\li(f)(x) :=\limsup_{y\to x} \frac{|f(x)-f(y)|}{d(x,y)}$ denotes the local Lipschitz constant of the function $f$.
Functions $f\in\mathcal F$ have a weak gradient, i.e.\ a function $|\nabla f|\in L^2(X,\mm)$ such that $\ch(f)=\frac{1}{2}\int_X |\nabla f|^2\dm$.

In what follows, we always assume that $X$ is infinitesimally Hilbertian, meaning that $\ch$ is a quadratic form. By polarisation of $\mathcal E(f):= 2\ch(f)$ we get a strongly local Dirichlet form $(\mathcal E, D(\mathcal E))$ on $L^2(X,\mm)$, where $D(\mathcal E):= \mathcal F$. The domain is then a Hilbert space with norm $\| f \|_{\mathcal E}^2 := \|f\|_{L^2(X,\mm)}^2+ \mathcal E(f)$. Thanks to the exponential integrability \eqref{eq:expint}, the Cheeger energy is quasi-regular, cf. \cite{Savare}*{Thm. 4.1}.  

Given an open subset $Y\subset X$ with $\mm(\partial Y)=0$, restricting to functions which vanish on $Z:=X\setminus Y$ quasi-everywhere, we get another Dirichlet form, corresponding to homogeneous Dirichlet ``boundary values'' on $Z$:  
  \begin{equation} 
    \begin{dcases} \label{eq:def_E_0}
   D(\mathcal E^0) := \{ f\in D(\mathcal E) \,\big|\, \tilde f=0 \text{ quasi-everywhere on } Z \}, \\
  \mathcal E^0(f) := \mathcal E(f) \text{ for } f\in D(\mathcal E^0), 
    \end{dcases}
  \end{equation}
where $\tilde f$ is the quasi-continuous representative of $f$.

By general Dirichlet form theory,  a symmetric, strongly continuous contraction semigroup on $L^2(X,\mm)$ is associated with each Dirichlet form. Thus we have a semigroup $(P_t)_{t>0}$ associated with $(\mathcal E, D(\mathcal E))$ and another one $(P_t^0)_{t>0}$ associated with $(\mathcal E^0, D(\mathcal E^0))$. They are related to the Dirichlet forms in the following way: For functions $f,g\in L^2(X,\mm)$ define the approximated forms $\mathcal E_t, \mathcal E_t^0:L^2(X,\mm)\times L^2(X,\mm)\to \mr$ by
  \begin{align*}
   \mathcal E_t(f,g):=& -\frac{1}{t} \int_X g (P_tf-f) \dm, \\
   \mathcal E_t^0(f,g):=& -\frac{1}{t} \int_X g (P_t^0f-f) \dm. 
  \end{align*}
Then we can recover the corresponding Dirichlet form in the following way (see \cite{FOT}*{Lemma 1.3.4}):
 \begin{equation} \begin{dcases}
   D(\mathcal E) = \left\{ f\in L^2(X,\mm) \,\Big|\, \lim_{t\to 0} \mathcal E_t(f,f) <\infty \right\}, \\
   \mathcal E(f,g) =\lim_{t\to 0} \mathcal E_t(f,g), \; \text{for } f,g\in D(\mathcal E).
  \end{dcases} \label{thm:approx_dirichlet_FOT}
 \end{equation}
Further, for $f\in L^2(X,\mm)$ the map $t\mapsto \mathcal E_t(f,f)$ is non-increasing and non-negative.  The same is true for $P_t^0$ and  $(\mathcal E^0, D(\mathcal E^0))$.

\subsection{Heat flows}\label{subsec:heat_flows}

Let us clarify the different heat flows. We have the ``usual'' heat flow and the one with Dirichlet boundary values, and to each a corresponding ``dual'' flow for measures.

\subsubsection*{Heat flow $P_t$ for functions on $X$.}
The heat flow $(t,u_0)\mapsto u_t=P_tu_0$ is defined by means of the semigroup in $L^2(X,\mm)$ corresponding to the Dirichlet form $(\mathcal E, D(\mathcal E))$.

\subsubsection*{Heat flow $\scrp_t$ for probability measures on $X$.}
From now on we additionally assume that $(X,d,\mm)$ is an $\rcd{K}{\infty}$ space. In this case, there is a Brownian motion $(B_t,\PP_x)$ on $X$ and corresponding to this a Markov kernel $p_t(x,A)= \PP_x(B_t\in A)$ (and even a heat kernel), all corresponding to the Dirichlet form $\mathcal E$, see \cite{AGMR}*{Sections 7.1, 7.2}. We use it to define the heat flow for probability measures: for $\mu\in \Pz(X)$ let 
  \[ \scrp_t\mu (A) := \int_X p_t(x,A) \dd\mu(x). \]
This coincides with the $\operatorname{EVI}_K$-flow of the entropy in $(\mathcal P_2(X),W_2)$.
Since the Brownian motion is connected to the Dirichlet form $\mathcal E$ uniquely, we get the following formula for the heat flow on functions through the Markov kernel
  \[ P_tf(x) = \int_X f(y) p_t(x,\dd y). \] 
The heat semigroups $P_t$ and $\scrp_t$ are dual in the following sense: For $f\colon X\to \mr$ bounded Borel, and $\mu\in \Pz(X)$ we have
  \begin{align} \int_X P_tf(x) \dd\mu(x) = \int_X \int_X f(y)p_t(x,\dd y) \dd\mu(x) = \int_X f(y) \int_X p_t(x,\dd y) \dd\mu(x) = \int_X f(y) \dd\scrp_t\mu(y). \label{eq:P_dual}\end{align}
  
The same applies to the heat flows $\hat P_t$ and $\hat\scrp_t$ on $\hat X$ (to be discussed in detail in the next section) and the equivalent flow $\tilde\scrp_t$ on $\tilde{\mathcal P}(Y|X)$, defined by means of the isometry introduced in Lemma \ref{lem:phipsi}.

 \subsubsection*{Heat flow with Dirichlet boundary values on $Y$.
 }

Let $Y\subset X$ be open 
and with $\mm(\partial Y)=0$.
Let us define a stopping time 
  \[ \tau_Z := \inf\{t>0 \,\big|\, B_t\in Z\}, \]
where as before $Z:=X\setminus Y$. Then we can define a Markov kernel
  \[ p_t^0(x,A) := \PP_x(B_t\in A,\, t < \tau_Z). \]
Note that we use Fukushima's convention that a Markov kernel is a \emph{sub}probability on $X$, in particular $p_t^0(x,A)\leq p_t(x,A)$. 
This Markov kernel is associated to the Dirichlet form $(\mathcal E^0,D(\mathcal E^0))$ given by \eqref{eq:def_E_0},
 see \cite{FOT}*{Thm. 4.4.2}. 
With this we can define the heat flows for bounded Borel functions $f\colon X\to \mr$ and measures $\mu\in \Pz^{sub}(X)$ as
  \[ P_t^0f(x) := \int_X f(y) p_t^0(x,\dd y) \]
and 
  \[ \scrp_t^0\mu (A) := \int_X p_t^0(x,A) \dd\mu(x). \]
They also satisfy the duality relation \eqref{eq:P_dual}.
  
  \begin{remark}
  With the help of the Markov kernels, all of these heat flows of measures can be extended to signed, finite Borel measures.
  \end{remark}

\section{Gluing} \label{sec:gluing}

In this section we glue together a finite number of copies of an open subset in a metric measure space ``along the boundary''. We will identify the Cheeger energy and the heat semigroup of the glued space in terms of the original objects.

Beginning with Alexandrov in the 40s, gluing has been studied in connection with curvature bounds a number of times, but mostly in Alexandrov spaces, see \cite{Alexandrov}*{``Verheftungssatz'' Kap. IX, \textsection 3}, \cite{Pogorelov}*{Chapter I, \textsection 11}, \cite{Perelman}*{\textsection 5}, \cite{Petrunin}*{Theorem 2.1}, \cite{Kosovskii}*{Theorem 1.1}. More recently, Schlichting \cites{SchlichtingThesis,SchlichtingArticle} applied the method of \cite{Kosovskii} to show preservation of various curvature bounds (among them Ricci curvature) on manifolds in an \emph{approximate} sense which we will use later to give the Riemannian case as an example. 
In \cite{Paulik}, metric measure spaces supporting Dirichlet forms are glued together. There is also a very recent preprint by Rizzi which shows that gluing does not preserve the dimension in the measure-contraction property \cite{Rizzi}. Apart from curvature bounds, the doubling of manifolds with boundary has also been applied by other communities to produce a related manifold without boundary, see for instance \cite{Atiyah_Bott}.

\subsection{Gluing of metric measure spaces} \label{subsec:gluing}

Take an open subset $Y\subset X$ and denote $Z:= X\setminus Y$. Fix a number $k\in\mn$. We now consider $k$ copies of $X$, denoted by $X^1,\dots,X^k$. 
We will identify these spaces with the original one via maps $\iota_i:X\to X^i, i=1,\dots,k$, which send points $x\in X$ to the corresponding points in $X^i$.
Each $X^i$ is equipped with the metric $d_i:=d\circ (\iota_i^{-1}, \iota_i^{-1})$ and the measure $\mm^i:={\iota_i}_{\#}\mm$, but in this section we usually suppress the indices and write $d$ and $\mm$ on every $X^i$. 
Let $Y^i:= \iota_i(Y),\, Z^i:= \iota_i(Z)$. We define an equivalence relation by identifying the points in the $Z^i$'s:
  \[ X^i\ni x \sim y\in X^j \;\;\; :\Leftrightarrow \;\;\; \left(i=j \textbf{ and } x=y\right) \text{ or } \left(\iota_i^{-1}(x)\in Z \textbf{ and }  \iota_i^{-1}(x) = \iota_j^{-1}(y) \right). \]
The \emph{$k$-gluing of $X$ along $Z$} is now obtained as the quotient of the disjoint union of the $X^i$ under this equivalence relation
  \[ \hat X := \left( \bigsqcup_{i=1}^k X^i \right)/\sim. \]
We can view $X^i$ as a subset of $\hat X$, since the canonical map $\sqcup_{i} X^i\to\hat X$ restricted to $X^i$ is injective. 
In the following, we will also make use of the partition 
  \[ \hat X = \left( \bigsqcup_{i=1}^k Y^i \right) \sqcup Z .\]
Define a metric $\hat d:\hat X\times \hat X\to \mr$ by
  \[ \hat d(x,y) := \begin{dcases}
                     \inf_{p\in Z} \left( d_i(x,\iota_i(p)) + d_j(\iota_j(p),y) \right) , & \text{ if } x\in X^i, y\in X^j, i\neq j \\
                     d(x,y) , & \text{ otherwise}.
                    \end{dcases}
 \]
As a measure we use
$\hat\mm:= \frac{1}{k}\sum_{i=1}^k \mm^i$, 
meaning that for a Borel set $A\subset \hat X$, we consider the restrictions to the copies and set
  \[ \hat\mm (A):= \frac{1}{k}\sum_{i=1}^k \mm^i(A\cap X^i). \]
This turns $\hat X$ into a metric measure space. \newline
For the special case of gluing together only two copies, we call the resulting space the \emph{doubling of} $Y$ in $X$, and as indices we will use $i\in\{+,-\}$.

\begin{Prop} \label{prop:X_hat_mms}
The space $(\hat X,\hat d)$ is a complete and separable length space, and the measure $\hat \mm$ is Borel.

If additionally $X$ is geodesic and $Z$ is proper (i.e.\! all closed balls are compact), then $\hat X$ is geodesic. 
\end{Prop}

\begin{proof}
 The metric properties are shown in \cite{Bridson_Haefliger}*{p.67f, Lemma 5.24}. 
\end{proof}

The metric properties directly transfer to the Wasserstein space, see for instance \cite{VIL09ii}.

\begin{Cor} \label{cor:W_hat_metric}
 For $p\in [1,\infty)$, the Kantorovich-Wasserstein metric $\hat W_p$ obtained from $\hat d$ is a complete, separable length metric on $\mathcal P_p(\hat X)$
\end{Cor}

Now we introduce some notation for dealing with functions on $\hat X$. 
For us it will be useful to consider the functions $u_i:X^i\to\mr$ given by
 $u_i:=u|_{X^i}$.
We consider the mean value
  $\bar u\colon X\to \mr, \;\;\bar u:= \frac{1}{k} \sum_{i=1}^k u_i\circ \iota_i$
and the ``mean free'' functions
  \[ \kringel{u}_i \colon X\to \mr,\;\;\; \accentset{\circ}{u}_i := u_i\circ \iota_i-\bar u. \]
Observe that since the $u_i$ all coincide on $Z$, the $\kringel u_i$ are zero \emph{everywhere} on $Z$. Also, we have 
	\begin{equation} \sum_{i=1}^k \kringel u_i =0. \label{eq:sumkringel0} \end{equation}
	
\noindent {\bf Notation:} During the proof of Lemma \ref{lem:P_GL_semigroup} we will start to 
simplify notation, by mostly omitting the identification maps $\iota_i$. Whenever a function $u_i$ now gets an argument from $X$, it is understood as $u_i\circ \iota_i$ and similar for $\overline u, \kringel{u}_i$ with $\iota_i^{-1}$. 

Let $(\chhat, \hat{\mathcal F})$ denote the Cheeger energy of the space $(\hat X, \hat d, \hat\mm)$. 

\begin{Lemma} \label{lem:X_hat_inf_hilbert}
 The space $\hat X$ is infinitesimally Hilbertian and for every $u\in \hat{\mathcal F}$, the functions $u_i\circ \iota_i$ are in $\mathcal F$ and 
  \[ \chhat(u) = \frac{1}{k}\sum_{i=1}^k \ch(u_i\circ \iota_i). \]
\end{Lemma}

\begin{proof}
 This follows directly from the locality property \eqref{thm:locality} of weak gradients by applying it to the open sets $Y^i$ and $Z^\circ$ (which can be found in \cite{AGS12RIEM}*{Thm. 4.19}):

Given a complete, separable metric space equipped with a Borel measure $(W,d_W,\mm_W)$, and an open subset $\Omega\subset W$ with $\mm_W(\partial\Omega)=0$, we have that the restriction of a function $f\in D(\ch^W)$ to $\overline\Omega$ is a function in $D(\ch^{\overline \Omega})$, and
\begin{equation} |\nabla  (f|_{\overline\Omega})|_{\overline\Omega} = (|\nabla f|_{W})|_{\overline\Omega} \;\;\;\mm\text{-a.e. in } \overline\Omega.  \label{thm:locality}\end{equation}

\end{proof}
In particular, we get a Dirichlet form $(\ehat, D(\ehat))$ on $\hat X$ by polarizing $\ehat(u):= 2\chhat(u)$ and setting $D(\ehat):= \hat{\mathcal F}$.

\begin{Lemma} \label{lem:domains}
 If $u\in D(\ehat)$, then $\bar u\in D(\mathcal E)$ and $\kringel u_i\in D(\mathcal E^0),\, i=1,\dots,n$.
\end{Lemma}

\begin{proof}
 Being in $D(\ehat)$ means $\chhat(u)<\infty$. By the previous lemma, this implies
  \[ \sum_{i=1}^k \frac{1}{k}\ch(u_i\circ \iota_i) = \chhat(u) <\infty. \]
 Since each term is non-negative, $\ch(u_i\circ \iota_i)<\infty$ for every $i=1,\dots,k$. Thus $u_i\circ \iota_i\in D(\mathcal E)$ and also the linear combination $\bar u\in D(\mathcal E)$.
 
 The other assertion follows from the fact that all the $u_i$'s coincide on $Z$.    
\end{proof}

Now we are going to define a semigroup on $\hat X$ and we will show that it actually is the one corresponding to $\ehat$.
\begin{Def} \label{def:P_GL_semigroup}
 The \emph{glued semigroup} $P_t^{GL}: L^2(\hat X,\hat\mm)\to L^2(\hat X,\hat\mm)$ is defined by
  \[ P_t^{GL}u(x) := P_t\bar u(\iota_i^{-1}(x)) + P_t^0\kringel u_i(\iota_i^{-1}(x)), \,\,\text{ if } x\in X^i ,\,\, i=1,\dots,k\\. \]
 Also, define the \emph{approximated glued Dirichlet form} $\mathcal E_t^{GL} : L^2(\hat X,\hat\mm)\times L^2(\hat X,\hat\mm) \to \mr$,
  \[ \mathcal E_t^{GL}(u,v) := -\frac{1}{t}\int_{\hat X} v (P_t^{GL}u-u) \dd\hat\mm. \]
\end{Def}

\begin{remark}
 Observe that $P_t^{GL}$ is well-defined, since $u_i=u_j$ on $Z$ for every $i,j=1,\dots,k$.
\end{remark}

\begin{Lemma} \label{lem:P_GL_semigroup}
 $(P_t^{GL})_{t>0}$ is a symmetric, strongly continuous contraction semigroup on $L^2(\hat X,\hat\mm)$. In particular, there is a corresponding Dirichlet form $(\mathcal E^{GL}, D(\mathcal E^{GL}))$ connected to $P_t^{GL}$ via
  \[ \begin{dcases}
   D(\mathcal E^{GL}) = \left\{ u\in L^2(\hat X,\hat\mm) \,\Big|\, \lim_{t\to 0} \mathcal E_t^{GL}(u) <\infty \right\} \\
   \mathcal E^{GL}(u,v) =\lim_{t\to 0} \mathcal E_t^{GL}(u,v), \; \text{for } u,v\in D(\mathcal E^{GL}).
  \end{dcases}\]
\end{Lemma}

\begin{proof}
 {\it Symmetry:} We use that $P_t$ and $P_t^0$ are symmetric with respect to $\mm$:
 \begin{align*}
  \int_{\hat X} u P_t^{GL}v \dd\hat\mm =& \sum_{i=1}^k \frac{1}{k} \int_{X^i} u_i \left((P_t\bar v)\circ \iota_i^{-1} + (P_t^0\kringel v_i)\circ \iota_i^{-1}\right) \dm^i \\
    =& \sum_{i=1}^k \frac{1}{k} \int_{X} \bar v P_t(u_i\circ \iota_i) + \kringel v_i P_t^0(u_i\circ \iota_i) \dm \\
    =& \sum_{i,j=1}^k \frac{1}{k^2} \int_{X} (v_j\circ \iota_j) P_t(u_i\circ \iota_i) + (v_i\circ \iota_i) P_t^0(u_i\circ \iota_i) - (v_j\circ \iota_j) P_t^0(u_i\circ \iota_i) \dm \\
    =& \sum_{i,j=1}^k \frac{1}{k^2} \int_{X} (v_j\circ \iota_j) P_t(u_i\circ \iota_i) + (v_j\circ \iota_j) P_t^0(u_j\circ \iota_j) - (v_j\circ \iota_j) P_t^0(u_i\circ \iota_i) \dm \\
    =& \sum_{j=1}^k \frac{1}{k} \int_{X} (v_j\circ \iota_j) \frac{1}{k}\sum_{i=1}^k P_t(u_i\circ \iota_i) + (v_j\circ \iota_j) \left(P_t^0(u_j\circ \iota_j) - \frac{1}{k}\sum_{i=1}^k P_t^0(u_i\circ \iota_i)\right) \dm \\
    =& \sum_{j=1}^k \frac{1}{k} \int_{X} (v_j\circ \iota_j) (P_t\bar u + P_t^0 \kringel u_j ) \dm 
    = \int_{\hat X} vP_t^{GL}u \dd\hat\mm.
 \end{align*}
From now on we will apply the abuse of notation introduced before. This is in order to improve readability.
 
 {\it Semigroup property:} 
 First observe that on $X^i$ we have $P_0^{GL}u=P_0\bar u + P_0^0\kringel u_i= \bar u + u_i-\bar u=u$.
 Denote $v:= P_t^{GL}u$. Then $v_i = P_t\bar u + P_t^0 \kringel u_i$. Now on $X^i$
 \begin{align*}
  P_s^{GL}P_t^{GL}u =& P_s^{GL}v = P_s\bar v + P_s^0\kringel v_i 
    = \frac{1}{k}\sum_{j=1}^k P_s v_j + P_s^0v_i - \frac{1}{k}\sum_{j=1}^k P_s^0v_j \\
    =& \frac{1}{k}\sum_{j=1}^k P_s (P_t\bar u + P_t^0 \kringel u_j) + P_s^0(P_t\bar u + P_t^0 \kringel u_i) - \frac{1}{k}\sum_{j=1}^k P_s^0(P_t\bar u + P_t^0 \kringel u_j) \\
    =& \frac{1}{k} \sum_{j=1}^k P_{s+t}\bar u + \underbrace{\frac{1}{k} \sum_{j=1}^k P_sP_t^0\kringel u_j}_{=0} + P_s^0P_t\bar u + P_{s+t}^0 \kringel u_i 
      - \frac{1}{k}\sum_{j=1}^k P_s^0P_t\bar u - \underbrace{\frac{1}{k}\sum_{j=1}^k P_{s+t}^0 \kringel u_j}_{=0} \\
    =& P_{s+t}\bar u + P_{s+t}^0 \kringel u_i 
    = P_{s+t}^{GL}u,
 \end{align*}
where we used \eqref{eq:sumkringel0}. 

 {\it Contraction:} To show the contraction property in $L^2(\hat X,\hat\mm)$, we first show that $P_t^{GL}$ is Markovian (i.e.\! positivity preserving and $L^\infty$-contractive in $L^2\cap L^\infty$). By symmetry of $P_t^{GL}$, we also get $L^1$-contractivity. Using the Riesz-Thorin interpolation theorem, we finally get contractivity in $L^2$. 
 
 Let $u\in L^2\cap L^\infty (\hat X,\hat\mm)$ with $0\leq u\leq 1$. Then also $0\leq u_i,\bar u\leq 1$. Then, on $X^i$,
  \[ P_t^{GL}u = P_t\bar u + P_t^0\kringel u_i \leq P_t\bar u+ P_t \kringel u_i = P_t u_i \leq 1. \]
 For the other side, we have to show $P_t^{GL}u\geq 0$, which is equivalent to
  \[ P_t^0 \bar u \leq P_t\bar u + P_t^0u_i. \]
 But this holds true because $P_t^0f\leq P_tf$ for every $f\in L^2$, and $P_t^0 u_i\geq 0$.
 
 Now we use that $L^1$ is a subspace of the dual of $L^\infty$. For $u\in L^1\cap L^2(\hat X,\hat\mm)$, consider the bounded, linear functional $\ell: L^\infty(\hat X,\hat\mm)\to \mr,\,  \ell(v):= \int_{\hat X} v P_t^{GL}u\dd\hat\mm$. The dual space norm of $\ell$ coincides with the $L^1$-norm of $P_t^{GL}u$, thus
  \begin{align*}
   \|P_t^{GL}u\|_{L^1(\hat X)} =& \sup_{\|v\|_{L^\infty(\hat X)}\leq 1} \int_{\hat X} vP_t^{GL}u \dd\hat\mm =  \sup_{\|v\|_{L^\infty(\hat X)}\leq 1} \int_{\hat X} P_t^{GL}v u \dd\hat\mm \\
    \leq&  \sup_{\|v\|_{L^\infty(\hat X)}\leq 1} \int_{\hat X} vu \dd\hat\mm = \|u\|_{L^1(\hat X)}.
  \end{align*}
 Here we used the symmetry of $P_t^{GL}$ and the $L^\infty$-contractivity.
 
 Hence $P_t^{GL}$ is a contraction in $L^1\cap L^2$ and also in $L^\infty\cap L^2$. By the Riesz-Thorin interpolation theorem, it is then also a contraction in $L^2$.
 
 {\it Strong continuity:} This follows directly from the strong continuity of $P_t$ and $P_t^0$:
  \begin{align*}
   \|P_t^{GL}u-u\|_{L^2(\hat X)}^2 =& \int_{\hat X} \left( P_t^{GL}u-u \right)^2 \dd\hat\mm = \sum_{i=1}^k \frac{1}{k}\int_{X^i} \left( P_t\bar u + P_t^0\kringel u_i - u_i \right)^2\dm^i \\
    =& \sum_{i=1}^k \frac{1}{k} \int_{X} \left( P_t\bar u -\bar u + P_t^0\kringel u_i - \kringel u_i \right)^2\dm \\
    \leq & \sum_{i=1}^k \frac{2}{k} \int_{X} \left( P_t\bar u -\bar u \right)^2 + \left(P_t^0\kringel u_i - \kringel u_i \right)^2\dm \\
    =& \sum_{i=1}^k \frac{2}{k} \left( \|P_t\bar u- \bar u\|_{L^2(X)}^2 + \|P_t^0\kringel u_i-\kringel u_i\|_{L^2(X)}^2 \right) 
     \longrightarrow 0 
  \end{align*}
  as $ t\to 0$.
\end{proof}

\begin{Lemma} \label{lem:approxE}
For every $u,v\in L^2(\hat X,\hat\mm)$:
  \begin{equation}  \mathcal E_t^{GL}(u,v) =  \mathcal E_t(\bar u,\bar v) + \frac{1}{k}\sum_{i=1}^k \mathcal E_t^0(\kringel u_i,\kringel v_i). \label{eq:approxE} \end{equation}
\end{Lemma}

\begin{proof}
 We just compute
 \begin{align*}
   \mathcal E_t^{GL}(u,v) =& -\frac{1}{t} \int_{\hat X} v\left(P_t^{GL}u-u\right) \dd\hat\mm \\
    =& -\sum_{i=1}^k \frac{1}{kt}\int_{X^i} v_i\left( P_t\bar u + P_t^0\kringel u_i-u_i \right) \dm^i 
    = -\sum_{i=1}^k \frac{1}{kt}\int_{X} v_i\left( P_t\bar u - \bar u + P_t^0\kringel u_i-\kringel u_i \right) \dm \\
    =& -\frac{1}{t} \int_{X} \bar v\left( P_t\bar u-\bar u \right)\dm - \sum_{i=1}^k \frac{1}{k} \int_X v_i\left( P_t^0 \kringel u_i- \kringel u_i \right) \dm 
      + \underbrace{\sum_{i=1}^k \frac{1}{k} \int_X \bar v\left( P_t^0 \kringel u_i- \kringel u_i \right) \dm}_{=0 \text{ by } \eqref{eq:sumkringel0}}\\
    =& \mathcal E_t(\bar u,\bar v) + \frac{1}{k}\sum_{i=1}^k \mathcal E_t^0(\kringel u_i,\kringel v_i). 
  \end{align*}
\end{proof}

\begin{Lemma} \label{lem:domains2}
If $u\in D(\mathcal E^{GL})$, then $\bar u\in D(\mathcal E)$ and $\kringel u_i\in D(\mathcal E^0), \, i=1,\dots,k$. 
\end{Lemma}

\begin{proof}
 By definition and \eqref{eq:approxE},
  \[ \infty > \mathcal E^{GL}(u) = \lim_{t\to 0} \mathcal E_t^{GL}(u) = \lim_{t\to 0}\left( \mathcal E_t(\bar u,\bar v) + \frac{1}{k}\sum_{i=1}^k \mathcal E_t^0(\kringel u_i,\kringel v_i) \right). \]
 Since the sum converges and every term is non-negative and non-decreasing as $t\to 0$, the terms converge and we can interchange sum and limit to get
  \[ \infty > \mathcal E^{GL}(u) = \lim_{t\to 0} \mathcal E_t(\bar u,\bar v) + \frac{1}{k}\sum_{i=1}^k \lim_{t\to 0} \mathcal E_t^0(\kringel u_i,\kringel v_i) =  \mathcal E(\bar u,\bar v) + \frac{1}{k}\sum_{i=1}^k \mathcal E^0(\kringel u_i,\kringel v_i). \]
\end{proof}

Now we come to the main theorem of this section, which identifies the semigroup $P_t^{GL}$ with the heat semigroup $\hat P_t$ associated to $\ehat.$
\begin{Thm}\label{thm:E_GL=E_hat}
The semigroups $P_t^{GL}$ and $\hat P_t$ coincide on $L^2(\hat X,\hat\mm)$ .
\end{Thm}

\begin{proof}
We will proof that the Dirichlet forms $(\mathcal E^{GL}, D(\mathcal E^{GL}))$ and $(\ehat, D(\ehat))$ coincide.
 Let $u,v\in D(\ehat)$. By Lemma \ref{lem:approxE},
  \[ \mathcal E_t^{GL}(u,v) =  \mathcal E_t(\bar u,\bar v) + \frac{1}{k}\sum_{i=1}^k \mathcal E_t^0(\kringel u_i,\kringel v_i). \]

 By Lemma \ref{lem:domains}, $\bar u,\bar v \in D(\mathcal E)$ and $\kringel u_i,\kringel v_i\in D(\mathcal E^0)$, so that we can take the limit $t\to 0$. This yields
  \begin{align*}
   \mathcal E^{GL}(u,v) =& \lim_{t\to 0} \mathcal E_t^{GL}(u,v) 
    = \lim_{t\to 0} \left( \mathcal E_t(\bar u,\bar v) + \frac{1}{k}\sum_{i=1}^k \mathcal E_t^0(\kringel u_i,\kringel v_i) \right) \\
    =& \mathcal E(\bar u,\bar v) + \frac{1}{k}\sum_{i=1}^k \mathcal E^0(\kringel u_i,\kringel v_i) 
    = \mathcal E(\bar u,\bar v) + \frac{1}{k}\sum_{i=1}^k \mathcal E(\kringel u_i,\kringel v_i) \\
    =& \mathcal E(\bar u,\bar v) + \frac{1}{k}\sum_{i=1}^k \mathcal E(u_i-\bar u,v_i-\bar v) 
    = \frac{1}{k} \sum_{i=1}^k \mathcal E(u_i,v_i) 
    = \ehat(u,v),
  \end{align*}
 where we used that $\mathcal E$ is an extension of $\mathcal E^0$.
 This also shows that $D(\ehat)\subset D(\mathcal E^{GL})$. The other direction works with the same argument but using Lemma \ref{lem:domains2} instead.
\end{proof}

\subsection{Identification of $\tild{\mathcal P}(Y|X)$ and $\mathcal P(\hat X)$} \label{subsec:identification}

We will show how the space of charged measures $\tild{\mathcal P}(Y|X)$ can be identified with the space of probability measures on the glued space, $\mathcal P(\hat X)$. Since we only look at two copies of $Y\subset X$, we index the different copies by $Y^+$ and $Y^-$ instead of the numerical indices in the previous subsection. Still, $Z:=X\setminus Y$ and $\hat X = \left( X^+\sqcup X^- \right)/\sim$. As we are dealing now with measures which are not equal on the different copies of $X$, in this section we do keep track of the identification maps $\iota_i, i\in\{+,-\}$. Every subset used in this section is assumed to be a Borel-measurable set in the space it is taken from.

\begin{Lemma}\label{lem:phipsi}
 The maps $ \Phi : \tild{\mathcal P}(Y|X) \to \mathcal P(\hat X)$ and $\Psi: \mathcal P(\hat X) \to \tild{\mathcal P}(Y|X)$, given by
  \[  \Phi((\sigma^+,\sigma^-))(A):= \sigma^+(\iota_+^{-1}(A\cap Y^+)) + \sigma^-(\iota_-^{-1}(A\cap Y^-)) + \sigma^+(\iota_+^{-1}(A\cap Z)) + \sigma^-(\iota_-^{-1}(A\cap Z))\]
  for  $A\subset \hat X$,
and 
  \[  \Psi(\hat \sigma)^i(B):= \hat\sigma(\iota_i(B)\cap Y^i) +\frac{1}{2}\hat\sigma(\iota_i(B)\cap Z)\]
  for   $B\subset X,\, i\in\{+,-\}$, 
respectively, are inverse to each other
and
 isometries between $(\tild{\mathcal P}_p(Y|X), \tild W_p)$ and $(\mathcal P_p(\hat X), \hat W_p)$ for each $p\in [1,\infty)$, where $\hat W$ denotes the Kantorovich-Wasserstein metric on $\mathcal P(\hat X)$.
\end{Lemma}

The \emph{proof} is straightforward and left to the reader.

The isometry allows to deduce a representation of the heat flow of charged measures in terms of the heat flows of their effective and total measures.
  
\begin{Lemma}\label{lem:P_t_measure}
Let $\sigma\in \tilde\Pz(Y|X)$. Then
  \[ \tilde\scrp_t\sigma = \left( \scrp_t\frac{\sigma^++\sigma^-}{2} + \scrp_t^0\frac{\sigma^+-\sigma^-}{2}, \scrp_t\frac{\sigma^++\sigma^-}{2} - \scrp_t^0\frac{\sigma^+-\sigma^-}{2} \right). \]
\end{Lemma}

\begin{proof}
We do the calculation in the equivalent setting of the doubled space $\hat X$.
Let $\hat\sigma\in \mathcal P(\hat X)$. Then
\begin{align*}
 \int_{\hat X} u \dd\hatp_t\hat\sigma =& \int_{\hat X} \hat P_tu \dd\hat\sigma \\
  =& \int_{X^+} \Big(P_t\frac{u^++u^-}{2} + P^0_t\frac{u^+-u^-}{2}\Big)\dd\sigma^+ + \int_{X^-} \Big(P_t\frac{u^++u^-}{2} - P^0_t\frac{u^+-u^-}{2}\Big)\dd\sigma^- \\
  =& \int_{X^+}\frac{1}{2}u^+\dd\scrp_t\sigma^+ + \int_{X^+}\frac{1}{2}u^-\dd\scrp_t\sigma^+ + \int_{X^+}\frac{1}{2}u^+\dd\scrp_t^0\sigma^+ - \int_{X^+}\frac{1}{2}u^-\dd\scrp_t^0\sigma^+ \\
  &+ \int_{X^-}\frac{1}{2}u^+\dd\scrp_t\sigma^- + \int_{X^-}\frac{1}{2}u^-\dd\scrp_t\sigma^- - \int_{X^-}\frac{1}{2}u^+\dd\scrp_t^0\sigma^- + \int_{X^-}\frac{1}{2}u^-\dd\scrp_t^0\sigma^- \\
 =& \int_{X^+} u^+ \dd\left(\scrp_t\frac{\sigma^++\sigma^-}{2} + \scrp_t^0\frac{\sigma^+-\sigma^-}{2}\right) + \int_{X^-} u^- \dd\left(\scrp_t\frac{\sigma^++\sigma^-}{2} - \scrp_t^0\frac{\sigma^+-\sigma^-}{2}\right)
\end{align*}
We relied heavily on the fact that we glue together copies of the \emph{same} space, making it possible to ``switch'' indices when necessary. 
\end{proof}

\begin{Lemma} \label{lem:ass=ent_hat_convex}
 Assumption \ref{ass} in $\tild{\mathcal P}_2(Y|X)$ is satisfied if and only if the entropy $\widehat \eent$ is convex in $\mathcal P_2(\hat X)$ (i.e.\! $\hat X$ is an $\rcd{K}{\infty}$ space).
\end{Lemma}

\begin{proof}
 Let $\hat\sigma\in \mathcal P_2(\hat X)$ with $\hat\sigma = \hat\xi\hat\mm$. We will show that the entropy of $\hat\sigma$ in $\mathcal P_2(\hat X)$ equals that of $\Psi(\hat\sigma)$ in $\tilde{\mathcal P}_2(Y|X)$ up to an additive constant, and then the result follows by Lemma \ref{lem:phipsi} and the fact that $K$-convexity is preserved if you add a constant to the functional.
 We have
  \begin{align*}
   \widehat\eent (\hat\sigma) =& \int_{\hat X} \hat \xi \log \hat\xi\dd\hat\mm \\
    =& \frac{1}{2}\int_{Y^+} \hat\xi|_{Y^+} \log \hat\xi|_{Y^+} \dm + \frac{1}{2}\int_{Y^-} \hat\xi|_{Y^-} \log \hat\xi|_{Y^-} \dm + \int_{Z} \hat\xi|_{Z} \log \hat\xi|_{Z} \dm \\
    =& \frac{1}{2}\int_{X^+} \hat\xi|_{X^+} \log \hat\xi|_{X^+} \dm + \frac{1}{2}\int_{X^-} \hat\xi|_{X^-} \log \hat\xi|_{X^+} \dm
  \end{align*}
 On the other hand, to compute $\widetilde\eent(\Psi(\hat\sigma))$, let us first identify the density of $\Psi(\hat\sigma)^i$ with respect to $\mm$: For a Borel-measurable set $A\subset X$
  \begin{align*}
   \Psi(\hat\sigma)^i (A) =& \hat\sigma (\iota_i(A)\cap Y^i) + \frac{1}{2}\hat\sigma(\iota_i(A)\cap Z) 
    = \int_{\iota_i(A)\cap Y^i} \dd\hat\sigma + \frac{1}{2}\int_{\iota_i(A)\cap Z} \dd\hat\sigma \\
    =& \int_{\iota_i(A)\cap Y^i} \frac{1}{2}\hat\xi \dd\mm + \frac{1}{2} \int_{\iota_i(A)\cap Z} \hat\xi\dm 
    = \frac{1}{2} \int_{\iota_i(A)\cap X^i} \hat\xi|_{X^i} \dm,
  \end{align*}
so that
  $\Psi(\hat\sigma)^i = \frac{1}{2} \left( \hat\xi|_{X^i}\circ \iota_i \right) \mm$.
Thus
  \begin{align*}
   \widetilde\eent(\Psi(\hat\sigma)) =& \eent(\Psi(\hat\sigma)^+) +\eent(\Psi(\hat\sigma)^-) \\
    =& \int_{X} \frac{1}{2} \left( \hat\xi|_{X^+}\circ \iota_+\right) \log \left(\frac{1}{2} \left( \hat\xi|_{X^+}\circ \iota_+\right)\right) \dm +  \int_{X} \frac{1}{2} \left( \hat\xi|_{X^-}\circ \iota_-\right) \log \left(\frac{1}{2} \left( \hat\xi|_{X^-}\circ \iota_-\right)\right) \dm \\
    =& \int_{X} \frac{1}{2} \left( \hat\xi|_{X^+}\circ \iota_+\right) \log \left(\left( \hat\xi|_{X^+}\circ \iota_+\right)\right) \dm + \int_{X} \frac{1}{2} \left( \hat\xi|_{X^+}\circ \iota_+\right) \log \left(\frac{1}{2}\right) \dm \\
    &+ \int_{X} \frac{1}{2} \left( \hat\xi|_{X^-}\circ \iota_-\right) \log \left(\left( \hat\xi|_{X^-}\circ \iota_-\right)\right) \dm + \int_{X} \frac{1}{2} \left( \hat\xi|_{X^-}\circ \iota_-\right) \log \left(\frac{1}{2}\right) \dm \\
      &+ \log \frac{1}{2} \underbrace{\int_X \frac{1}{2} \left( \hat\xi|_{X^+}\circ \iota_+\right) + \frac{1}{2} \left( \hat\xi|_{X^-}\circ \iota_-\right) \dm}_{=1} \\
    =& \widehat\eent (\hat\sigma) + \log\frac{1}{2}.
  \end{align*}
 \end{proof}

\section{Transportation (semi-)distances between subprobabilities}\label{sec:Wnull}

Let $(X,d)$ be a complete separable metric space and $Y\subset X$ be an open subset with $\emptyset\not= Y\not= X$.
Recall the definition of 
$L^p$-transportation semi-metric between subprobabilities  $\mu,\nu\in\Pz^{sub}(Y)$:
 \begin{align*}
 W^0_p(\mu,\nu):=& \inf\Big\{\tild W_p(\sigma,\tau) \,\Big|\, \sigma,\tau\in\tild\Pz(Y|X), \sigma^0=\mu, \tau^0=\nu\Big\}
\\
=&\inf\Big\{\tild W_p\big( (\mu+\rho,\rho), (\nu+\eta,\eta)\big) \,\Big|\, \rho,\eta\in\Pz^{sub}(X), 
(\mu+2\rho)(X)=1, (\nu+2\eta)(X)=1\Big\}. \nonumber
\end{align*}

\begin{proof}[Proof of Lemma \ref{lem:hatW_metric}]
This is an immediate consequence of the isometry between $\tild{\mathcal P}_p(Y|X)$ and $\mathcal P_p(\hat X)$, together with Lemma \ref{prop:X_hat_mms}.
\end{proof}

Every coupling of the charged probability measures $(\mu+\rho,\rho)$ and $(\nu+\eta,\eta)$ induces a decomposition of each of the involved measures into three parts. This leads to another, more detailed description of the transportation costs from above.
\begin{Lemma} \label{lem:remark_decomposition}
Let $\mu,\nu\in \mathcal P_p^{sub}(Y)$. Then
 \begin{eqnarray}
   W^0_p(\mu,\nu)^p
    &=&
  \inf\Big\{ W_p(\mu_1,\nu_1)^p+W_p(\mu_2,\eta_1^+)^p+W_p^*(\mu_3,\eta_1^-)^p \notag\\
  &&\qquad
    +W_p(\rho^+_1,\nu_2)^p+W_p(\rho^+_2,\eta_2^+)^p+W_p^*(\rho^+_3,\eta_2^-)^p \notag\\
  &&\qquad
    +W^*_p(\rho^-_1,\nu_3)^p+W^*_p(\rho^-_2,\eta_3^+)^p+W_p(\rho^-_3,\eta_3^-)^p
    \,\Big|\, \label{eq:W0_decomposition}\\
  &&\qquad\
    \mu=\mu_1+\mu_2+\mu_3, \rho=\rho_1^++\rho_2^++\rho_3^+=\rho_1^-+\rho_2^-+\rho_3^-,(\mu+2\rho)(X)=1, \notag\\
  &&\qquad\
    \nu=\nu_1+\nu_2+\nu_3, \ \eta=\eta_1^++\eta_2^++\eta_3^+=\eta_1^-+\eta_2^-+\eta_3^-, \ (\nu+2\eta)(X)=1     \Big\}. \notag 
\end{eqnarray}
The decompositions implicitly require the coupled measures to have the same mass, so for instance $\mu_1(X)=\nu_1(X)$ etc.
\end{Lemma}
The \emph{proof} consists in using again the isometry between $\tilde{\mathcal P}_p(Y|X)$ and $\hat{\mathcal P}_p(\hat X)$ and disintegrating the appearing measures.
In the case $p=1$, a more explicit description is possible.

\begin{Lemma} \label{lem:rep_W0}
 For all $\mu,\nu\in\mathcal P_p^{sub}(Y)$ and $p=1$
  \begin{align*} 
   {W^0_p}(\mu,\nu)^p = \inf \Big\{ & W_p(\mu_1,\nu_1)^p + W_p^*(\mu_0)^p + W^*_p(\nu_0)^p \,\Big|\, \\
      & \mu=\mu_1+\mu_0, \nu=\nu_1+\nu_0, (\mu+\nu_0)(X)\leq 1, (\nu+\mu_0)(X)\leq 1 \Big\}.
  \end{align*}
  Moreover, the $\le$-inequality holds for all $p\in[1,\infty)$ if $(X,d)$ is a length space.
\end{Lemma}

\begin{proof}
 The ``$\leq$''-direction follows from the previous Lemma by choosing the decomposition $\rho_3^+=\eta_2^-=\rho_2^-=\eta_3^+=0$ and $\rho_2^+=\eta_2^+=\rho_3^-=\eta_3^-$, so that 
 	\begin{align*}
 	{W}^0_p(\mu,\nu)^p \leq & \inf \big\{ W_p(\mu_1,\nu_1)^p + W_p(\mu_2,\eta_1^+)^p + W^*_p(\mu_3,\eta_1^-)^p +  W_p(\rho_1^+,\nu_2)^p + W^*_p(\rho_1^-,\nu_3)^p \,\big|\, \\
 	&\qquad\qquad\qquad (\mu+2\nu_2)(X)\le1, \ (\nu+2\mu_2)(X)\le1 \big\}\\
 	\le&   \inf \big\{  W_p(\mu_1,\nu_1)^p + W^*_p(\mu_0)^p + W^*_p(\nu_0)^p \,\big|\, \ 
 	(\mu+\nu_0)(X)\le1, \ (\nu+\mu_0)(X)\le1 \big\}.
 	\end{align*}
For the second inequality, we used in the case $p=1$ simply the fact that $\rho_1^+=\rho_1^-, \eta_1^+=\eta_1^-$ and 
$$\inf_{\eta_1^+,\,  \mu_2+\mu_3=\mu_0} \Big[ W_1(\mu_2,\eta_1^+)+W_1^*(\eta_1^+,\mu_3)\Big]\le \frac12 W_1^*(\mu_0,\mu_0)= W_1^*(\mu_0)$$ by choosing $\eta_1^+=\mu_2=\mu_3=\frac12\mu_0$.

The case $p>1$ requires a more sophisticated argumentation using optimal transport in the glued space $\hat X=(X\setminus Y)\cup Y^+\cup Y^-$. We freely switch between  equivalent representations in 
$(\tild{\mathcal P}_p(Y|X), \tild W_p)$ and in $(\mathcal P_p(\hat X), \hat W_p)$.
Assume for simplicity that $(X,d)$ is geodesic. (For general length spaces, one has to use approximation arguments based on almost geodesics.)
Given a $\tilde W_p$-geodesic $(\sigma_t)_{t\in[0,1]}$ connecting $\sigma_0:=(\mu_0,0)$ and $\sigma_1:=(0,\mu_0)$, we decompose it into two
$\tilde W_p$-geodesics $(\sigma'_t)_{t\in[0,1]}$ and $(\sigma''_t)_{t\in[0,1]}$ such that
$\tilde W_p(\sigma'_0,\sigma'_1)=\tilde W_p(\sigma''_0,\sigma''_1)=\frac12\tilde W_p(\sigma_0,\sigma_1)$
and $\sigma'_{1/2}(Y^-)=\sigma''_{1/2}(Y^+)=0$. (Actually, it suffices that
$\tilde W_p(\sigma'_0,\sigma'_1)\ge\frac12\tilde W_p(\sigma_0,\sigma_1)$
and $\sigma'_{1/2}(Y^-)=0$.)
  Choosing $\mu_2=(\sigma'_0)^+$, $\mu_3=(\sigma'_1)^-$,  and $\eta_1^+=(\sigma'_{1/2})^+$ then yields
\begin{eqnarray*}
	\inf_{\eta_1^+,\,  \mu_2+\mu_3=\mu_0} \Big[ W_p(\mu_2,\eta_1^+)^p+W_p^*(\eta_1^+,\mu_3)^p\Big]&\le& 
  W_p\big((\sigma'_0)^+,(\sigma'_{1/2})^+\big)^p+W_p^*\big((\sigma'_{1/2})^+,(\sigma'_1)^+\big)^p\\
  &=&
 \tilde W_p\big(\sigma'_0,\sigma'_{1/2}\big)^p+\tilde W_p\big(\sigma'_{1/2},\sigma'_1\big)^p
  =2^{1-p}\tilde W_p\big(\sigma'_0,\sigma'_{1}\big)^p\\
  &\le &2^{-p}\tilde W_p\big(\sigma_0,\sigma_{1}\big)^p=2^{-p}W^*_p(\mu_0,\mu_0)^p=W^*_p(\mu_0)^p.
 \end{eqnarray*}

 \bigskip
 
   To prove the ``$\geq$''- inequality, we assume for simplicity that minimizers in the definition of $W_1^0$ exist. This is for instance the case when $X$ is compact. For the general case one has to work with almost-minimizers. \newline 
   Let subprobabilities $\mu$ and $\nu$ be given as well as $\rho$ and $\eta$ with $(\mu+2\rho)(X)=1, (\nu+2\eta)(X)=1$ 
 such that
\begin{eqnarray*} 
  {W}^0_1(\mu,\nu)&=&\tilde W_1\big((\mu+\rho,\rho),  (\nu+\eta,\eta)\big)\\
  &=& \hat W_1\big(\mu+\rho+\rho',\nu+\eta+\eta'\big)
\end{eqnarray*}
where for the last identity we switched to the picture of the glued space $\hat X=(X\setminus Y)\cup Y^+\cup Y^-$
with subprobabilities $\mu,\nu,\rho,\eta$ on the ``upper'' sheet $(X\setminus Y)\cup Y^+$
and their copies $\rho',\eta'$ on the ``lower'' sheet $(X\setminus Y)\cup Y^-$. We further assume for the moment that all masses are rational numbers. 

Given $\varepsilon>0$, choose $n,n_1,n_2\in\N$ and $x_i,y_i,u_i,v_i\in X^+$ for $i=1,\ldots,n$ such that
$$W_1(\mu,\mu_n)\le\varepsilon,\quad
W_1(\nu,\nu_n)\le\varepsilon,\quad W_1(\rho,\rho_n)\le\varepsilon,\quad W_1(\eta,\eta_n)\le\varepsilon$$
for 
$$\mu_n=\frac1n\sum_{i=1}^{n-2n_1}\delta_{x_i},\quad
\nu_n=\frac1n\sum_{i=1}^{n-2n_2}\delta_{y_i},\quad
\rho_n=\frac1n\sum_{i=1}^{n_1}\delta_{u_i},\quad
\eta_n=\frac1n\sum_{i=1}^{n_2}\delta_{v_i}.
$$
Hence also $W_1(\rho',\rho'_n)\le\varepsilon$, $W_1(\eta',\eta'_n)\le\varepsilon$ for
$\rho'_n=\frac1n\sum_{i=1}^{n_1}\delta_{u'_i}$,
$\eta'_n=\frac1n\sum_{i=1}^{n_2}\delta_{v'_i}
$ with $u'_i= \iota_-\circ\iota_+^{-1}(u_i)$ and similarly for $v'_i$.
(To avoid ambiguity, we may assume that the sets $\{x_i\}$ and $ \{y_i\}$ are disjoint form each other.)
In particular we have $\frac{n_1}{n}=\rho(X)$ and so on.

Now fix a $\hat W_1$-optimal coupling $q_n$ of $\mu_n+\rho_n+\rho'_n$ and $\nu_n+\eta_n+\eta'_n$
 on $\hat X$. Without restriction, we can choose this coupling $q_n$ as a matching (i.e.  it does not split mass), that is,
$$q_n=\frac1n\sum_{\xi\in Q_n}\delta_\xi$$
with suitable $Q_n\subset Z\times W$ where $Z=\{x_i\}\cup\{u_i\}\cup\{u'_i\}$ and $W=\{y_i\}\cup\{v_i\}\cup\{v'_i\}$.
Now consider chains of (pairwise disjoint) pairs in $Q_n$ with either initial points or endpoints of subsequent pairs being conjugate to each other. These chains of maximal length will be of the form
\begin{itemize}
\item[Case 1:]
$(z_1,w_1), (z'_2,w_1'), (z_2,w_2), (z'_3, w'_2),\ldots,(z'_{k},w'_{k-1}),(z_{k},w_k)$
\item[Case 2:]
$(z_1,w_1), (z'_1,w_2'), (z_2,w_2), (z'_2, w'_3),\ldots,(z'_{k-1},w'_{k}),(z_{k},w_k)$
\item[Case 3:]
$(z_1,w_1), (z'_2,w_1'), (z_2,w_2),\ldots,(z'_{k},w'_{k-1})$
with $z'_k\not=z'_1$
\item[Case 4:]
$(z_1,w_1), (z'_1,w_2'), (z_2,w_2),\ldots,(z'_{k-1},w'_{k})$
with $w'_k\not=w'_1$
\item[Case 5:]
$(z_1,w_1), (z'_2,w_1'), (z_2,w_2),\ldots,(z'_{1},w'_{k-1})$
\item[Case 6:]
$(z_1,w_1), (z'_1,w_2'), (z_2,w_2),\ldots,(z'_{k-1},w'_{1})$
\end{itemize}
with $z_i,z'_i\in Z$, $w_i,w'_i\in W$ and $z\mapsto z'$ denoting the ``conjugation map'' which switches between upper and lower sheet. In particular, $(z')'=z$.

Now let us have a closer look on the previous six cases of chains of maximal length.
\begin{itemize}
\item[Case 1:] Maximality implies $z_1\in\{x_i\}$ and $w_k\in \{y_i\}$ whereas all the other points inbetween $w_i,w'_i,z_i,z'_i\in \{u_i\}\cup\{u'_i\}\cup\{v_i\}\cup\{v'_i\}$. The transportation cost associated with this chain is at least
$$ \hat d(z_1,w_1) + \hat d(w'_1,z'_2) + \dots + \hat d(z_k,w_k)\ge \hat d(z_1,w_k) = d(z_1, w_k) $$
and thus is bounded from below by the cost of the direct transport between the endpoints.

Denote by $X_1\subset \{x_i\}$ the set of $z_1$ in case 1 and by 
$Y_1\subset \{y_i\}$ the set of $w_k$. Let $$\mu_n^1=\frac1n\sum_{x\in X_1}\delta_x, \quad \nu_n^1=\frac1n\sum_{y\in Y_1}\delta_y.$$ 
Then the transport costs arising from all pairs contained in any chain of case 1 is bounded from below by
$W_1\big(\mu_n^1,\nu_n^1\big)$.

\item[Case 2:] This is just a relabeling of case 1 with indices running in reverse order. No additional costs arise.

\item[Case 3:] Here, maximality implies $z_1\in\{x_i\}$ and also $z_k'\in\{x_i\}$. Thus at least one of the pairs in the chain consists of points from two different sheets. 
Thus with the triangle inequality on $\hat X$, we conclude that the cost of this chain is at least 
$d^*(z_1,z'_k)$.

Denote by $X_0\subset \{x_i\}$ the set of $z_1$   in case 3. Note that this set coincides with the set  of $z_k'$ (just by reverting the chain) -- but for  calculating the cost induced by the coupling $q_n$, only one of the pairs $(z_1,z'_k)$ and $(z'_k,z_1)$ has to be taken into account. 

Let $$\mu_n^0=\frac1n\sum_{x\in X_0}\delta_x.$$ 
Then the transport costs arising from all pairs contained in any chain of case 3 is bounded from below by
$\frac12W^*_1\big(\mu_n^0,\mu_n^0\big)$.

\item[Case 4:] Similarly, here we conclude  $w_1\in\{y_i\}$ as well as $w_k'\in\{y_i\}$ and that the cost of the chain is at least 
$d^*(w_1,w'_k)$. Denote by $Y_0\subset \{x_i\}$ the set of $w_1$   in case 4 and set
 $$\nu_n^0=\frac1n\sum_{y\in Y_0}\delta_y.$$ 
Then the transport costs arising from all pairs contained in any chain of case 4 is bounded from below by
$\frac12W^*_1\big(\nu_n^0,\nu_n^0\big)$.

\item[Case 5:] The cyclic chains in this case will produce superfluous costs which will vanish for optimal choices of measures $\rho_n,\eta_n$.
That is, 0 is the best lower estimate for the 
transport costs arising from all pairs contained in any chain of case 5.
This infimum will be attained by chains of length $k=2$ of the form $(z_1,w_1), (z'_1,w'_1)$ with $z_1=w_1$.

\item[Case 6:] This is a cyclic permutation of case 5. No additional costs arise.

\end{itemize}
Summarizing, we obtain
$$\hat W_1\big(\mu_n+\rho_n+\rho'_n,\nu_n+\eta_n+\eta'_n\big)\ge 
W_1\big(\mu_n^1,\nu_n^1\big)+
\frac12W^*_1\big(\mu_n^0,\mu_n^0\big)+\frac12W^*_1\big(\nu_n^0,\nu_n^0\big).$$

Now for given $\varepsilon$ and $n$, the decomposition $\mu_n=\mu_n^1+\mu_n^0$ induces via the optimal coupling of $\mu_n$ and $\mu$ a decomposition $\mu=\mu^1+\mu^0$ such that
$$W_1(\mu^1,\mu^1_n)\le\varepsilon, \quad W_1(\mu^0,\mu^0_n)\le\varepsilon.$$
Similarly, for $\nu_n=\nu_n^1+\nu_n^0$ and $\nu=\nu^1+\nu^0$.
Thus we finally obtain
\begin{eqnarray}
W_1^0(\mu,\nu)&=&
\hat W_1\big(\mu+\rho+\rho',\nu+\eta+\eta'\big) \notag\\
&\ge& 
\hat W_1\big(\mu_n+\rho_n+\rho'_n,\nu_n+\eta_n+\eta'_n\big)-6\varepsilon \notag\\
&\ge&W_1\big(\mu_n^1,\nu_n^1\big)+
\frac12W^*_1\big(\mu_n^0,\mu_n^0\big)+\frac12W^*_1\big(\nu_n^0,\nu_n^0\big)-6\varepsilon \notag\\
&\ge&W_1\big(\mu^1,\nu^1\big)+
\frac12W^*_1\big(\mu^0,\mu^0\big)+\frac12W^*_1\big(\nu^0,\nu^0\big)-10\varepsilon. \label{eq:W_^0_rep_rational}
\end{eqnarray} 
Since $\varepsilon>0$ was arbitrary, this proves the claim.

\medskip

 For the general case of real masses, one can approximate Borel measures by sums of Dirac measures (with rational masses) in the weak topology. By continuity of $\tilde W_1, W_1$ and $W_1^*$ with respect to weak convergence, one can apply the rational case and go to the limit in \eqref{eq:W_^0_rep_rational}.
\end{proof}

      \medskip

 \begin{proof}[Proof of Lemma \ref{lem:rep_W0_1}] Assertions (i) and (ii) are the content of the previous Lemma. The proof for the decomposition in assertion (iv) is straightforward. For the vanishing of the $W_p^\dagger$-term in the case $p=1$ note that in this case $[d'(x,\partial)+d'(x,\partial)]^p=d'(x,\partial)^p+d'(x,\partial)^p$ whereas in general only the $\ge$ inequality holds.
 
 Assertion (iii) will follow from combining assertion (iv), Lemma  \ref{lem:Wstern_rand} and Theorem \ref{thm:W_sharp-W_prime}(i). 
 \end{proof}

In the case of a length space  $X$, the annihilation cost $W_1^*(\mu)$ allows for an alternative characterization as $\inf\{ W_1(\mu,\xi): \xi\in{\mathcal P}(\partial Y)\}$ and, more generally, 
$$W_1^*(\mu,\nu)= \inf\{ W_1(\mu,\xi)+ W_1(\xi,\nu) \,\big|\, \xi\in{\mathcal P}(\partial Y)\}.$$
This is the content of Lemma \ref{lem:Wstern_rand}.

\begin{proof}[Proof of Lemma \ref{lem:Wstern_rand}]
We switch to the picture of two glued copies. Given $\mu,\nu\in \mathcal P(Y)$, consider them as $\mu\in \mathcal P(Y^+)$ and $\nu\in \mathcal P(Y^-)$ and fix a $\hat W_1$-optimal coupling $q$ of them.

To simplify the presentation, let us first discuss the argument if $\hat X$ is a  geodesic space.
Choose a measurable selection of connecting  $\hat d$-geodesics
$\Gamma: \hat X\times \hat X\to \mathrm{Geo}(\hat X)$.
For a geodesic $\gamma$ in $\hat X$ with $\gamma_0\in Y^+, \gamma_1\in Y^-$ define
$\alpha(\gamma)=\inf\{s: \gamma_s\not\in Y^+\}$
and $z(\gamma):=\gamma_{\alpha(\gamma)}$.
Finally, define a map $Y^+\times Y^-\to \partial Y$ by  $\mathcal Z =z\circ \Gamma$.  

Define a probability measure $\xi=\mathcal Z_\#q$ via push forward of the optimal coupling. Then this is a $\hat W_1$-intermediate point of $\mu$ and $\nu$.
Indeed, for the transport from $\mu$ to $\nu$, the pair $x\in Y^+, y\in Y^-$ contributes the cost $d^*(x,y)$.
The fraction $\alpha(x,y) \cdot d^*(x,y)$ contributes to the cost of
 the transport from $\mu$ to $\xi$. And the fraction   $(1-\alpha(x,y)) \cdot d^*(x,y)$ contributes to the cost of
 the transport from $\xi$ to $\nu$.
 
 Now let us discuss the general case of a length space $X$. Instead of geodesics, we now choose approximate $\hat d$-geodesics. With the same construction then $\xi$ will be an approximate $\hat W_1$-intermediate  point. This proves the claim in the case $p=1$.
 
 \medskip
 
 To prove the claim for $p>1$, for simplicity we assume that $X$ is compact. (This will guarantee the existence of the map $\Phi$ to be introduced below. Otherwise, one has to use approximation arguments.)

 For each $\xi\in{\mathcal P}(\partial Y)$ and each $W_p$-optimal coupling $q$ of  $\mu$ and $\xi$ 
  \[ W_p(\mu,\xi)^p = \int_{X\times X} d(x,y)^p \dd q(x,y) \geq \int_{X\times X} d'(x,\partial)^p \dd q(x,y) = W'_p(\mu,0)^p. \]

 To deduce the converse inequality, choose a measurable $\Phi: Y\to\partial Y$ such that for each $x\in Y$ the point $\Phi(x)$ is a minimizer of $z\mapsto d(x,z)$ on $\partial Y$. Define a probability measure $\xi=\Phi_\sharp \mu$. Then
  \[ W_p(\mu,\xi)^p \leq \int_X d(x,\Phi(x))^p \dd \mu(x) = \int_X d'(x,\partial)^p \dd \mu(x) = W'_p(\mu,0)^p. \]
 This proves that 
  \[ W_p'(\mu,0) = \inf\{ W_p(\mu,\xi) \,\big|\, \xi\in{\mathcal P}(\partial Y)\}. \]
 
 \medskip
 
 Moreover, the triangle inequality for $d^*$ implies
 that $W_p(\mu,\xi)+ W_p(\xi,\mu) \geq W^*_p(\mu,\mu)$ for all $\xi\in{\mathcal P}(\partial Y)$. Thus $W'_p(\mu,0)\ge W^*_p(\mu)$. An estimate in the other direction is obtained as follows
  \begin{align*}
   W_p^*(\mu)^p = 2^{-p} W_p^*(\mu,\mu)^p =&\, 2^{-p} \int_{X\times X} \left( \inf_{z\in X\setminus Y} \big( d(x,z)+ d(z,y) \big) \right)^p \dd q(x,y) \\
    \geq &\, 2^{-p} \int_{X\times X} \left( \inf_{z\in X\setminus Y} d(x,z) + \inf_{w\in X\setminus Y} d(w,y) \right)^p \dd q(x,y) \\ 
    \geq &\, 2^{1-p} \int_{X\times X} \left( \inf_{z\in X\setminus Y} d(x,z)\right)^p \dd q(x,y) = 2^{1-p} W_p'(\mu,0)^p,
  \end{align*}
 where $q$ denotes any $W^*_p$-optimal coupling of $\mu$ and $\mu$.
  \end{proof}

\begin{proof}[Proof of Theorem \ref{thm:W_sharp-W_prime}] 
(i) 
For simplicity of the presentation we assume that length minimizing geodesics exist. This is for instance the case when $Y'$ is geodesic. In this case there exist $W_1'$-geodesics which are supported on $d'$-geodesics. For the general case one has to work with almost-geodesics.

Recall that then $W_1'$ is a geodesic metric on  $\mathcal P^{sub}_1(Y)$
and that, according to Lemma \ref{lem:rep_W0_1}iv) and Lemma \ref{lem:Wstern_rand}, 
\begin{align*} 
    W_1'(\mu,\nu) = \inf \Big\{ & W_1(\mu_1,\nu_1) + W_1^*(\mu_0) + W^*_1(\nu_0) \,\Big|\,     \mu=\mu_1+\mu_0, \nu=\nu_1+\nu_0 \Big\}.
  \end{align*}
 for all subprobability measures $\mu,\nu\in \mathcal P^{sub}_1(Y)$.
 Together with Lemma \ref{lem:rep_W0_1}i) this implies $ W_1'(\mu,\nu)\le  W_1^0(\mu,\nu)$. In particular, $W_1^0$ does not vanish outside the diagonal. As $W_1^\flat$ is the biggest metric below $W_1^0$, we have $W_1' \leq W_1^\flat$. Using the fact that $W_1'$ is a \emph{geodesic}  metric, we thus get
  \begin{align*}
   W_1^\sharp (\mu,\nu) =& \inf_{\substack{\eta\colon \mu \leadsto \nu \\ W_1^\flat\text{-cont.}}} \sup_{0=s_0<\ldots<s_n=1} \sum_{i=1}^n W_1^\flat(\eta_{s_{i-1}}, \eta_{s_i}) \\
    \geq& \inf_{\substack{\eta\colon \mu \leadsto \nu \\ W_1^\flat\text{-cont.}}} \sup_{0=s_0<\ldots<s_n=1} \sum_{i=1}^n W_1'(\eta_{s_{i-1}}, \eta_{s_i}) \\
    \geq& \inf_{\substack{\eta\colon \mu \leadsto \nu \\ W_1'\text{-cont.}}} \sup_{0=s_0<\ldots<s_n=1} \sum_{i=1}^n W_1'(\eta_{s_{i-1}}, \eta_{s_i}) \
    = \ W_1'(\mu,\nu).
  \end{align*}

 To prove the converse inequality, given $\mu,\nu\in\mathcal P^{sub}(Y)$, let $(\eta_s')_{s\in[0,1]}$ be a $W'_1$-geodesic connecting $\mu',\nu'$ in $\mathcal P(Y')$ which is supported on (constant-speed) $d'$-geodesics. Decompose this geodesic into two geodesics
 $\eta'_s=\eta'_{s,1}+\eta'_{s,0}$ where $(\eta_{s,1}')_{s\in[0,1]}$ is a $W'_1$-geodesic supported by $d'$-geodesics staying in $Y$ and $(\eta_{s,0}')_{s\in[0,1]}$ is a $W'_1$-geodesic supported by $d'$-geodesics passing through $\partial$.
 
 Now replace the latter by another curve  $(\tilde\eta_{s,0}')_{s\in[0,1]}$ with the same endpoints:
 $$\tilde\eta'_{s,0} :=\left\{
 \begin{array}{ll}
 (1-2s)\eta'_{0,0}+2s\eta'_{0,0}(Y')
\, \delta_\partial,& s\in[0,\frac12]\\
 (2s-1)\eta'_{1,0}+2(1-s)\eta'_{0,0}(Y')
\, \delta_\partial,\quad& s\in(\frac12,1].
 \end{array}\right.$$
 (Indeed, this is also a $W'_1$-geodesic since in the $L^1$-Wasserstein geometry also convex combinations are geodesics.)
 Consider 
  $\tilde\eta_{s}=\tilde\eta'_{s,0}\big|_Y+\eta_{s,1}$.
 This is a curve in $\mathcal P^{sub}(Y)$ which connects $\mu$ and $\nu$. Moreover, taking decompositions 
 	\[ \tilde\eta_s = \underbrace{(\eta_{s,1}+\tilde\eta'_{t,0}|_Y)}_{``\mu_1"} + \underbrace{2(t-s)\eta_{0,0}}_{``\mu_0"} \;\;\text{ and }\;\; \tilde\eta_t= \underbrace{(\eta_{t,1} + \tilde\eta'_{t,0}|_Y)}_{``\nu_1"} + \underbrace{0}_{``\nu_0"}\] 
 in Lemma \ref{lem:rep_W0_1} i) for $s\leq t\leq \frac{1}{2}$ and similar for the other cases, we get
 $$W^0_1(\tilde\eta_s,\tilde\eta_t)\le 
 \left\{
 \begin{array}{ll}
 |t-s|\cdot W_1(\eta_{0,1},\eta_{1,1})
 +2|t-s|\cdot W_1^*(\eta_{0,0}) , \quad& \mbox{for }s,t\le\frac12\\
  |t-s|\cdot W_1(\eta_{0,1},\eta_{1,1})
 +2|t-s|\cdot W_1^*(\eta_{1,0}) , \quad& \mbox{for }s,t\ge\frac12
 \end{array}\right.$$
 and thus
 $$L^\flat_1(\tilde\eta)\le W_1(\eta_{0,1},\eta_{1,1})+W_1^*(\eta_{0,0})+W_1^*(\eta_{1,0}) =W_1'(\mu,\nu)$$
 which finally implies $W_1^\sharp(\mu,\nu)\le W_1'(\mu,\nu) $.

   Since $W_1^\sharp$ is the length metric induced by $W_1^\flat$, one gets $W_1^\flat \leq W_1^\sharp$. 
 The other inequality is provided by the fact that $W_1^\flat$ is the biggest metric below $W_1^0$ and that $W_1^\sharp = W_1' \leq W_1^0$ by the above.
 
 \medskip

(ii) Now let us consider the case $p>1$. The idea is that locally (along a geodesic) the contribution of $W_p^\dagger$ is negligible, so that we can compare $W_p'$ and $W_p^\flat$ on a small scale and then carry it over to the induced length metrics. \newline 
Let subprobabilities $\mu,\nu$ be given as well as a $W'_p$-geodesic $(\eta'_t)_{t\in[0,1]}$ connecting the measures $\mu':=\mu+(1-\mu(Y))\delta_\partial$ and $\nu':= \nu + (1-\nu(Y))\delta_\partial$. By the continuity of $W_p'$ and $W_p^*$ with respect to weak convergence we can assume without loss of generality that $\mu$ and $\nu$ have compact supports and for $\eps>0$ small
	\[ \eta_t(Y) \leq 1-\eps \]
for all $t\in[0,1]$. Recall that the measures without primes are the restrictions to $Y$. We thus have $\eta_t(\partial) =0$, whereas $\eta_t'(\partial)\geq \eps$. Choose $\delta>0$ such that $\eta_t(B_\delta'(\partial))\leq \frac{\eps}{2}$. Let $\Pi$ be the probability measure on the space of $Y'$-geodesics such that $\eta_t'=(\ee_t)_\#\Pi$ (where $\ee_t$ is the evaluation map at time $t$), denote by $L$ the essential supremum of $d'(\gamma_0,\gamma_1)$ under $\Pi$, and let $\delta':=\frac{\delta}{L}$. \newline
We consider $\eta_s$ and $\eta_t$ for $|s-t|\leq \delta'$.
Using that $d^\dagger(x,y)^p \geq d'(x,\partial)^p+d'(y,\partial)^p$, we see that in the decomposition \eqref{W'-decompo} it is actually cheaper to annihilate mass at the boundary:
	\begin{align*}
	W_p'(\eta_s,\eta_t)^p = &\inf \Big\{ W_p(\eta_{s,1},\eta_{t,1})^p + W_p^\dagger(\eta_{s,2},\eta_{t,2})^p + W_p'(\eta_{s,0},0)^p + W_p'(\eta_{t,0},0)^p \,\Big|\, \\
	&\eta_s=\eta_{s,1} + \eta_{s,2} + \eta_{s,0}, \eta_{t} = \eta_{t,1} + \eta_{t,2} + \eta_{t,0}, (\eta_{s}+\eta_{t,0})(Y)\leq 1, (\eta_{t}+\eta_{s,0})(Y)\leq 1 \Big\} \\
	\geq &  \inf \Big\{  W_p(\eta_{s,1},\eta_{t,1})^p +  W'_p(\eta_{s,0}+ \eta_{s,2},0)^p + W'_p(\eta_{t,0} + \eta_{t,2},0)^p \,\Big|\\  
	 & \eta_s=\eta_{s,1}+\eta_{s,2} + \eta_{s,0}, \eta_t=\eta_{t,1}+\eta_{t,2} + \eta_{t,0}, (\eta_s+\eta_{t,0})(Y)\leq 1, (\eta_t + \eta_{s,0})(Y)\leq 1\Big\}.
	\end{align*}
Since $W^\dagger$ only occurs where $d^\dagger$ is smaller than $d$, its contribution comes from geodesics in $B_\delta'(\partial)$, so that by our choice of $\delta$ we know that $\eta_{s,2}(Y)= \eta_{s,2}(B_\delta'(\partial)) \leq \frac{\eps}{2}$ and the same for $\eta_{t,2}$. Hence for $\eps$ small enough we have $(\eta_s+(\eta_{t,2}+\eta_{t,0}))(Y)\leq 1$, so that $\eta_s=\eta_{s,1} + \tilde\eta_{s,0}$ with $\tilde\eta_{s,0}:=\eta_{s,0}+\eta_{s,2}$ is an admissible decomposition. In particular, the above inequality is an \emph{equality}.
Note that we cannot use this trick for $s=0,t=1$ because then the constraint would not be satisfied.
Thanks to Lemma \ref{lem:Wstern_rand} we thus have
	\begin{align*}
	W_p'(\eta_s,\eta_t)^p \geq & \inf \Big\{  W_p(\eta_{s,1},\eta_{t,1})^p +  W^*_p(\tilde\eta_{s,0})^p + W^*_p(\tilde\eta_{t,0})^p \,\Big|\\  
	&  \qquad\eta_s=\eta_{s,1}+\tilde\eta_{s,0}, \eta_t=\eta_{t,1}+\tilde\eta_{t,0}, (\eta_s+\tilde\eta_{t,0})(Y)\leq 1, (\eta_t+\tilde\eta_{s,0})(Y) \leq 1 \Big\}\\
	\geq& W^0_p(\eta_s,\eta_t)^p  \geq  W^\flat_p(\eta_s,\eta_t)^p.
	\end{align*}
Hence, the $W'_p$-length of the curve $(\eta_t)_{t\in[0,1]}$ dominates its $W^\flat_p$-length. This finally proves
$$W'_p(\eta_s,\eta_t)^p\ge W^\sharp_p(\eta_s,\eta_t)^p$$
for all $s,t$. For $s=0, t=1$, this yields the claimed upper estimate for $W_p^\sharp$.

\medskip

The lower estimate follows from assertion i) together with the facts that $W_1^\flat\le W_p^\flat$ (which is inherited from analogous inequalities for $\tilde W_p$ and in turn for $W^0_p$) and $W_p^\flat\le W_p^\sharp$.
\end{proof}

\begin{proof}[Proof of Proposition \ref{cor:W_sharp_vague_conv}] Boundedness of $X$, say $d(x,y)\le D$, implies that all the  $W^\sharp_p$-metrics are continuous w.r.t.\ to each other: $W_1^\sharp\le W^\sharp_p\le D^{1-1/p}\cdot (W_1^\sharp)^{1/p}$.
Thus it suffices to prove the claim for $p=1$.

Assume that $W^\sharp_1(\mu_n,\mu)\to 0$. Then  $W'_1(\mu'_n,\mu')\to 0$ and thus 
$\mu'_n$ converges to $\mu'$ weakly on $Y'$. This in turn obviously implies that  $\mu_n$ converges  to $\mu$ vaguely on $Y$.

Now conversely assume that $W^\sharp_1(\mu_n,\mu)\not\to 0$. By compactness of $Y'$, there will exist $\nu'\in \mathcal P(Y')$ such that -- after passing to a suitable subsequence -- 
$\mu'_n$ converges to $\nu'$ weakly on $Y'$. Let $\nu$ denote the restriction of $\nu'$ to $Y$. Obviously, $\nu\not=\mu$. (Otherwise, $W^\sharp_1(\mu_n,\mu)\to 0$.) Thus there exists $f\in {\mathcal C}_c(Y)$ with $\int fd\mu\not=\int fd\nu=\lim_{n}\int fd\mu_n$ and therefore
$\mu_n$ does not converge to $\mu$ vaguely on $Y$.
\end{proof}

The following simple estimate will make it possible to prove the continuity of $W_p^0$ with respect to weak convergence plus convergence of moments of subprobability measures.
  
\begin{Lemma} \label{lem:W0_estimate}
 Let $\mu, \nu \in \mathcal P^{sub}_1(Y)$ with $\mu(Y) \geq \nu(Y)$. Then, for any $z\in X\setminus Y$,
  \[ W_1^0(\mu,\nu) \leq \inf \left\{ W_1(\mu_1,\nu) + \int_{X} d(x,z) \dd \mu_0(x) \,\big|\, \mu=\mu_1+\mu_0, \mu_1(Y)=\nu(Y) \right\}. \]
\end{Lemma}

\begin{proof}
 Taking a decomposition such that $\nu_1=\nu, \nu_0=0$, Lemma \ref{lem:rep_W0_1} yields
 $W_1^0(\mu,\nu) \leq W_1(\mu_1,\nu) + W^*_1(\mu_0)$. 
 Using now
 \begin{align*}
   W_1^*(\mu_0,\mu_0) =& \inf_{q} \int_{X\times X} d^*(x,y) \dd q(x,y) \\
      \leq& \inf_{q} \int_{X\times X} \big[d(x,z) + d(z,y)\big] \dd q(x,y) 
    = 2 \int_{X} d(x,z) \dd\mu_0(x),
  \end{align*}
 the proof is complete. 
\end{proof}

\begin{lemma} \label{lem:W_sharp_weak_conv}
  Assume that $X$ is compact. Then for  $\mu^{(n)},\mu^*\in \mathcal P^{sub}(Y)$ the following are equivalent:
\begin{itemize}
\item[(i)] $\mu^{(n)}\to \mu^*$ weakly on $Y$
   \item[(ii)]  $W^0_p(\mu^{(n)},\mu^*)\to 0$ and $\mu^{(n)}(Y)\to \mu^*(Y)$
    \end{itemize}
\end{lemma}

\begin{remark} 
	Without assuming compactness in Lemma \ref{lem:W_sharp_weak_conv}, we are still able to get that $W_p^0(\mu^{(n)}, \mu^*) \to 0$ for $\mu^{(n)},\mu^*\in \mathcal P_p^{sub}(Y)$ if $\mu^{(n)}\to \mu^*$ weakly in $Y$ and $\int d(x,x_0)^p \dd\mu^{(n)}(x) \to \int d(x,x_0)^p \dd \mu^*(x)$ for some $x\in Y$.
\end{remark}

\begin{proof}[Proof of Lemma \ref{lem:W_sharp_weak_conv}] 
Assume $\mu^{(n)} \to \mu^*$ weakly on $Y$. It again suffices to prove the result for $p=1$.
We want to use Lemma \ref{lem:W0_estimate} to show continuity. In order to apply this lemma, we have to decompose the \emph{larger} measure. We will proceed in three steps. First we will consider only sequences $(\mu^{(n)})$ with
$\mu^{(n)}(Y)\geq \mu^*(Y)$

for all $n\in\mn$. Define $\lambda_n :=\frac{\mu^*(Y)}{\mu^{(n)}(Y)}$ and $\mu^{(n)}_1:= \lambda_n \mu^{(n)}$. Then $\mu^{(n)}_1(Y) = \mu^*(Y)$, $\lambda_n\to 1$, and for $f\in C_b^0$
  \begin{align*}
   \left| \int_Xf\dd\mu^{(n)}_1 - \int_X f \dd\mu^* \right| 
    \leq & \left| \int_X \lambda_n f \dd\mu^{(n)} - \int_X f\dd\mu^{(n)} \right| + \left| \int_X f \dd\mu^{(n)} - \int_X  f \dd\mu^* \right| \\
    =& |\lambda_n-1| \left| \int_X f \dd\mu^{(n)} \right| + \left| \int_X f \dd\mu^{(n)} - \int_X  f \dd\mu^* \right| 
    \longrightarrow 0.
  \end{align*}
 Hence, we have convergence in the Kantorovich-Wasserstein metric: $W_1(\mu^{(n)}_1, \mu^*) \to 0$.
Writing $\mu^{(n)}_0:= (1-\lambda_n) \mu^{(n)}$, by Lemma \ref{lem:W0_estimate} we finally have
  \[ {W^0_1}(\mu^{(n)},\mu^*) \leq W_1(\mu^{(n)}_1,\mu^*) + \int_X d(x,z) \dd\mu^{(n)}_0(x) \longrightarrow 0. \]
Now, for the case that $\mu^{(n)}(Y)\leq \mu^*(Y)$, let $\lambda'_n:= \frac{\mu^{(n)}(Y)}{\mu^{*}(Y)}$ and $\mu^*_{1,n} := \lambda'_n\mu^*$. Then $\mu^*_{1,n}(Y)= \mu^{(n)}(Y)$ and $\lambda'_n\to 1$. Given $f\in C_b^0$, by
  \[ \left| \int_X f\dd \mu^*_{1,n} - \int_X f \dd \mu^* \right| \leq |\lambda'_n-1| \left| \int_X f \dd \mu^* \right| \longrightarrow 0, \]
we see that $\mu^*_{1,n} \rightharpoonup \mu^*$. In a next step this yields
  \begin{align*}
   \left| \int_X f \dd \mu^*_{1,n} - \int_X f \dd\mu^{(n)} \right| \leq \left| \int_X f \dd \mu^*_{1,n} - \int_X f \dd\mu^{*} \right| + \left| \int_X f \dd \mu^* - \int_X f \dd\mu^{(n)} \right| \longrightarrow 0,
  \end{align*}
i.e.\! $\mu^*_{1,n}-\mu^{(n)} \rightharpoonup 0$. 
Hence, using again Lemma \ref{lem:W0_estimate}, we see that
  \[  {W^0_1}(\mu^{(n)},\mu^*) \leq W_1(\mu^{(n)},\mu^*_{1,n}) + \int_X d(x,z) \dd\mu^*_{0,n }(x) \longrightarrow 0 \]
Since a sequence converges if and only if every subsequence has a convergent subsequence, we now can conclude  that $a_n:= {W^0_1}(\mu^{(n)},\mu^*)$ converges to 0. Indeed, take a subsequence $a_{n_k}$. Then we can take a further subsequence $a_{n_{k_\ell}}$ such that either $\mu^{(n_{k_\ell})}(Y) \geq \mu^*(Y)$ for every $\ell\in\mn$, or $\mu^{(n_{k_\ell})}(Y) \leq \mu^*(Y)$ for every $\ell\in\mn$. But then the above ensures convergence of these subsequences to 0.

  \medskip
  
   Conversely, now assume that $\mu^{(n)}(Y)\to \mu^*(Y)$ and ${W^0_p}(\mu^{(n)},\mu^*)\to 0$. 
Let $\rho^{(n)},\eta^{(n)} \in {\mathcal P}^{sub}(X)$ such that $(2\rho^{(n)} +\mu^{(n)})(X)=1 = (2\eta^{(n)}+\mu^*)(X)$, and ${W^0_p}(\mu^{(n)},\mu^*) = \tilde W_p((\mu^{(n)}+\rho^{(n)}, \rho^{(n)}),(\mu^*+\eta^{(n)},\eta^{(n)}))$. 
Let $\mu^{(n_k)}$ be any subsequence and consider the corresponding subsequences $\rho^{(n_k)}, \eta^{(n_k)}$. 
Compactness of $\hat X$ implies that 
there exists a sub-subsequence $(n_{k_\ell})_\ell$ such that 
  \[ \eta^{(n_{k_\ell})}\rightharpoonup \eta^* \text{ and } \mu^{(n_{k_\ell})} \rightharpoonup \tilde\mu^* \text{ and } \rho^{(n_{k_\ell})}  \rightharpoonup \rho^* \]
  with suitable limits points $\eta^*, \tilde\mu^*, \rho^*$.
Then we have
  \begin{align*}
    \tilde W_p \left( (\tilde\mu^*+\rho^*, \rho^*),(\mu^*+\eta^*,\eta^*) \right) \leq & \ \tilde W_p \left( (\tilde\mu^*+\rho^*, \rho^*), (\mu^{(n_{k_\ell})}+\rho^{(n_{k_\ell})}, \rho^{(n_{k_\ell})}) \right) \\
      &+ \tilde W_p\left( (\mu^{(n_{k_\ell})}+\rho^{(n_{k_\ell})}, \rho^{(n_{k_\ell})}),(\mu^*+\eta^{(n_{k_\ell})},\eta^{(n_{k_\ell})}) \right) \\
      &+ \tilde W_p\left( (\mu^*+\eta^{(n_{k_\ell})}, \eta^{(n_{k_\ell})}),(\mu^*+\eta^*,\eta^*) \right) 
       \longrightarrow 0.
  \end{align*}
Hence $\rho^*=\eta^*$ and in particular $\tilde\mu^*=\mu^*$. This way we see that every subsequence of $\mu^{(n)}$ has a further subsequence which converges to $\mu^*$, so that also the whole sequence converges to $\mu^*$. 
\end{proof}

\section{Proofs for Subsections \ref{subsec:gradient_flow_perspective} \& \ref{subsec:gradient_estimate_bochner}} \label{sec:proofs}

\begin{proof}[Proof of Proposition \ref{evi}]
 This will follow from the identification with the glued space and the properties shown in Subsection \ref{subsec:gluing}, in particular Theorem \ref{thm:E_GL=E_hat}. Let us provide the details.
 
  i) Given $\sigma_0\in \tilde\Pz(Y|X)$, consider $\hat\sigma:= \Phi(\sigma_0) \in \Pz(\hat X)$, with the isometry $\Phi$ given in Lemma \ref{lem:phipsi}. Since $\hat X$ is an $\rcd{K}{\infty}$ space by Assumption \ref{ass} and Lemma \ref{lem:ass=ent_hat_convex}, the $\evik$-gradient flow $\hat\sigma_t\in \Pz(\hat X)$ starting in $\hat\sigma$ exists. Again by the identification of the entropies in Lemma \ref{lem:ass=ent_hat_convex}, the flow $\sigma_t:= \Psi(\hat\sigma_t)$ is the $\evik$-gradient flow of $\widetilde\eent$ in $\tilde\Pz(Y|X)$.
  
  ii) Let $\mu_0 \in \mathcal P^{sub}_2(X)$, and let $\sigma_0\in \tilde\Pz (Y|X)$ such that $\mu_0=\sigma_0^+-\sigma_0^-$ (such a $\sigma_0$ exists by definition of $\Pz^{sub}_2(X)$). Consider $\sigma_t := \tilde\scrp_t\sigma_0$. By Lemma \ref{lem:P_t_measure} we have
    \[ \sigma_t^+-\sigma_t^-=\scrp_t(\sigma_0^+-\sigma_0^-) = \scrp_t\mu_0. \]
    This also shows the independence of the chosen $\sigma_0$, as the right-hand side is independent of it. 
    
  iii) As in ii).
  
  iv) Let $\sigma_0\in \tilde\Pz_2(Y|X)$ and define $\mu_0:=\sigma_0^+-\sigma_0^-$ and $\nu_0:=\sigma_0^++\sigma_0^-$. Then, again by Lemma \ref{lem:P_t_measure},
    \begin{align*}
     \sigma_t = \tilde\scrp_t\sigma_0 =& \left( \scrp_t \frac{\sigma_0^++\sigma_0^-}{2} + \scrp_t^0\frac{\sigma_0^++\sigma_0^-}{2}, \scrp_t \frac{\sigma_0^++\sigma_0^-}{2} + \scrp_t^0\frac{\sigma_0^+-\sigma_0^-}{2} \right) \\
      =& \left( \scrp_t \frac{\mu_0}{2} + \scrp_t^0\frac{\nu_0}{2}, \scrp_t \frac{\mu_0}{2} + \scrp_t^0\frac{\nu_0}{2} \right) \\
      =& \left( \frac{\mu_t+\nu_t}{2}, \frac{\mu_t-\nu_t}{2} \right).
    \end{align*}  
 \end{proof}
 
 \begin{proof}[Proof of Proposition \ref{prop:W_contraction}]
  This is again a direct consequence of the identification, since by Assumption \ref{ass} the glued space is an $\rcd{K}{\infty}$ space and thus satisfies the desired Wasserstein contraction.
 \end{proof}
 
 \begin{proof}[Proof of Theorem \ref{thm:ass_implies_bochner}] 
 i) Under Assumption $\ref{ass}$, $\hat X$ is an $\rcd{K}{\infty}$ space and hence satisfies a gradient estimate with $p=2$. By \cite{Savare}*{Cor. 4.3} we have the improved gradient estimate for $p\in [1,2]$ and by Jensen's inequality one easily obtains the gradient estimate for $p>2$ from that. Now we take a function $f\in D(\mathcal E^0)$ and define 
    \begin{equation*}
     u:=\begin{cases}
         f, &\text{ on } X^+\\
         -f, & \text{ on } X^-.
        \end{cases}
    \end{equation*}
  Then $u\in D(\ehat)$ and $|\nabla u|=|\nabla f|$ on each $X^i$. Thus, inserting $u$ in the gradient estimate on $\hat X$ yields on the upper half $X^+$:
    \[ |\nabla P_t^0f|^p = |\nabla \hat P_tu|^p \leq e^{-pKt} \hat P_t|\nabla u|^p = e^{-pKt} P_t|\nabla f|^p. \]

 ii) This follows directly from the duality of the heat semigroups \eqref{eq:P_dual}.
 \end{proof}

 \begin{proof}[Proof of Theorem \ref{thm:halfspace}]
  (i)$\Rightarrow$(ii): Consider the doubling of $X$, $V:= \hat X$. Then we can view $Y$ as an open subset of $\hat X$ by identifying it with $Y^+$. Now define $\psi: V\to V$ as the ``mirror mapping'' 
    \[ \psi(x) := \begin{cases}
                   \iota_-\circ \iota_+^{-1} (x), & \text{if } x\in X^+ \\
                   \iota_+\circ \iota_-^{-1} (x), & \text{if } x\in X^-.
                  \end{cases}
 \]
 It is easy to see that $\psi$ is a measure-preserving isometry. Further, let $x\in X^+$ such that $\psi(x)=x$, i.e.\! $\iota_- \circ \iota_+^{-1}(x)=x$. This in particular means $x\in Z$ since for $x\in Y^+$ we would have $\iota_-\circ \iota_+^{-1}(x)\in Y^-$, which would contradict $\psi(x)=x\in Y^+$. Finally observe that $\psi(Y)=\psi(Y^+)=\iota_-(Y)=Y^-=V\setminus Y^+$. 
 
 (ii)$\Rightarrow$(iii): Take $i=1, V_1:=V$.
 
 (ii)$\Rightarrow$(i): Thanks to $\xi$, we can define a measure-preserving isometry $\varphi\colon (V,d_V,\mm_V) \to (\hat X,\hat d,\hat \mm)$ by mapping $Y\cong \tilde Y$ to $Y^+$, $\psi(Y)$ to $Y^-$ and $\partial Y$ to $Z=X\setminus Y\subset \hat X$, where $\psi$ is the map given in the definition of a halfspace. Since curvature-dimension conditions are preserved under measure-preserving isometries, $\hat X$ is an $\rcd{K}{\infty}$ space. Lemma \ref{lem:ass=ent_hat_convex} then tells us that Assumption \ref{ass} is satisfied.
 
 (iii)$\Rightarrow$(i): We want to show that $\hat X$ is an $\rcd{K}{\infty}$ space by using the local-to-global property. Given $x\in \partial Y$, choose $i$ such that $x\in X_i$. Then we can identify $(Y\cap X_i)^+\cup (Y\cap X_i)^-\subset \hat X$ with $\hat W_i\subset V_i$ via $\xi_i$. Given measures $\mu_0,\mu_1\in \Pz(\hat X)$ supported in $(Y\cap X_i)^+\cup (Y\cap X_i)^-$, then $\nu_\ell:=(\xi_i)_\# \mu_\ell \in \Pz(V_i), \ell=0,1,$ are supported in $\hat W_i$. Since $V_i$ is an $\rcd{K}{\infty}$ space, there is a geodesic $\nu_t\in \Pz(V_i)$ connecting $\nu_0$ and $\nu_1$ such that the entropy $\operatorname{Ent}_{\mm_{V_i}}$ is convex. Pulling back this curve via $\mu_t:=(\xi_i^{-1})_\#\nu_t$ provides us with a geodesic in $\Pz(\hat X)$ such that $\widehat{\eent}$ is convex. Combining this convex optimal transport near the boundary (i.e.\! the gluing edge) together with the local $\operatorname{RCD}$ property of $X$ (and hence $X^+$ and $X^-$), we have that $\hat X$ is a local $\rcd{K}{\infty}$ space and by the local-to-global property also an $\rcd{K}{\infty}$ space. 
 \end{proof}

Let us finally come to the proof of Theorem \ref{riem}. 
When interested in curvature properties, gluing together Riemannian manifolds is a delicate issue, since in general the glued Riemannian metric will only be continuous and so one cannot define the curvature tensors. 

\begin{Thm} \label{thm:doublemanifold}
 Let $(M,g)$ be a complete, $n$-dimensional Riemannian manifold with Ricci curvature bounded below by $K\in \mr$. Let $Y\subset M$ be an open, bounded, convex subset with a smooth, compact boundary, equip it with the Riemannian distance $d$ and volume measure $\mm$, and write $X:=\overline Y$. Then the 2-gluing of $(X,d,\mm)$ along $\partial Y$, denoted by $(\hat X,\hat d,\hat\mm)$, is an $\rcd{K}{n}$ space.
\end{Thm}

\begin{proof}
First observe that the gluing of Riemannian manifolds yields a continuous Riemannian metric
  \[ \hat g(p) = \begin{cases}
               g_+(p), &\text{ if } p\in Y^+ \\
               g_-(p), &\text{ if } p\in Y^-,
              \end{cases}
 \]
whose Riemannian distance and volume measure are $d_{\hat g}=\hat d$ and $\mm_{\hat g}=2\hat \mm$ in terms of our metric gluing.

 By convexity, the submanifold $Y$ satisfies the same lower bound on the Ricci curvature. A result of Schlichting \cites{SchlichtingArticle,SchlichtingThesis} now ensures that there is a sequence of \emph{smooth} Riemannian metrics $\hat g_\eps$ on the glued manifold $\hat X$ converging to $\hat g$ uniformly as $\eps\to 0$ and such that
  \[ \ric_{\hat g_\eps} \geq (K-\eps). \]
 Thus we get a sequence of smooth, compact metric measure spaces $(\hat X,d_{\hat g_\eps},\mm_{\hat g_\eps})$ which satisfy the $\rcd{K-\eps}{n}$ condition. The stability of this condition under measured Gromov-Hausdorff convergence together with the convergence result in the following lemma completes the proof.
\end{proof}

\begin{Lemma} \label{lem:unif_conv_GH}
 Let $(g_\eps)_{\eps>0}$ be a sequence of smooth Riemannian metrics and $g$ a continuous Riemannian metric on a compact, smooth manifold $\mathfrak M$. If $g_\eps \to g$ uniformly as $\eps\to 0$, then $(\mathfrak M, d_\eps, \mm_\eps) \to (\mathfrak M, d, \mm)$ in the measured Gromov-Hausdorff sense, where $d_\eps,\mm_\eps$ and $d,\mm$ are the distance functions and volume measures obtained by $g_\eps$ and $g$, respectively.
\end{Lemma}
This seems to be well-known. We leave its straightfoward \emph{proof} to the reader.

\begin{proof}[Proof of Theorem \ref{riem}]
 As a Riemannian manifold with lower Ricci curvature bound $K$, $M$ is an $\rcd{K}{\infty}$ space. As a \emph{convex} subset, also $\overline Y$ with the restricted distance and measure is an $\rcd{K}{\infty}$ space. Now Assumption \ref{ass} is satisfied by identification of the entropies in Lemma \ref{lem:ass=ent_hat_convex}, since the doubling of the manifold is an $\rcd{K}{\infty}$ space by Theorem \ref{thm:doublemanifold}. 
\end{proof}

 %--------------------------------------------------------------------------


\end{document}